\documentclass[11p,reqno]{amsart}
\usepackage{amsmath}
\usepackage{amsfonts}
\usepackage{amssymb}
\usepackage{graphicx}
\usepackage{color}
\newtheorem{theorem}{Theorem}[section]
\newtheorem*{theorem*}{Theorem}
\newtheorem{lemma}{Lemma}[section]
\newtheorem{corollary}[theorem]{Corollary}
\newtheorem{proposition}{Proposition}[section]

\newtheorem{definition}[theorem]{Definition}

\newtheorem{conjecture}[theorem]{Conjecture}

\newtheorem{remark}[theorem]{Remark}

\def\i{\sqrt{-1}}
\def\Ric{\text{Ric}}

\def\i{\sqrt {-1}}

\def\aint{\frac{\ \ }{\ \ }{\hskip -0.4cm}\int}

\def\id{\operatorname{id}}
\def\Ric{\operatorname{Ric}}

\numberwithin{equation}{section}

\begin{document}
	\title[The fundamental group and the positivity]{The fundamental group, rational connectedness and  the positivity of K\"ahler manifolds}

\author{Lei Ni}\thanks{The research is partially supported by  ``Capacity Building for Sci-Tech Innovation-Fundamental Research Funds".  }
\address{Lei Ni. Department of Mathematics, University of California, San Diego, La Jolla, CA 92093, USA}
\email{leni@math.ucsd.edu}



\begin{abstract} First we confirm  a conjecture   asserting that any compact K\"ahler manifold $N$ with $\Ric^\perp>0$ must be simply-connected by applying a new  viscosity  consideration to Whitney's  comass of $(p, 0)$-forms. Secondly  we prove the projectivity and the  rational connectedness of a K\"ahler manifold of complex dimension $n$ under the condition  $\Ric_k>0$ (for some $k\in \{1, \cdots, n\}$, with $\Ric_n$ being the Ricci curvature), generalizing a well-known result of Campana, and independently of Koll\'ar-Miyaoka-Mori, for the Fano manifolds. The proof utilizes both the above comass consideration and  a second variation consideration of \cite{Ni-Zheng2}.  Thirdly, motivated by $\Ric^\perp$ and the classical work of Calabi-Vesentini \cite{CV}, we propose  two new curvature notions.  The cohomology vanishing $H^q(N, T'N)=\{0\}$ for any $1\le q\le n$ and  a deformation rigidity result are obtained under these new curvature conditions. In particular they are verified for all classical K\"ahler C-spaces with $b_2=1$. The new conditions provide viable candidates for a curvature characterization of homogenous K\"ahler manifolds related to a generalized Hartshone conjecture.
\end{abstract}

\maketitle

\section{Introduction}

K\"ahler manifolds bridge the Riemannian manifolds, complex manifolds and complex algebraic manifolds. It  avails analytic and geometric techniques in the study of algebraic manifolds via the GAGA principle. The first general result on the {\it projectivity} of a high dimensional K\"ahler manifold $(N, h)$, i.e. being able to be realized as a holomorphic submanifold in some projective space $\mathbb{P}^K$, was obtained by Kodaira \cite{Kod}. Kodaira proved that the projectivity is equivalent to  the existence of an {\it integral} K\"ahler form $\omega_h$ in $H^2(N, \mathbb{Z})$. It was also shown that this cohomological condition is  equivalent to the existence of a {\it positive line bundle} $L$. A line bundle is positive  means that  there exists a Hermitian metric $h$ on $L$ such that the Chern-form of $(L, h)$ is positive. From the Riemannian geometric point of view the most natural way of associating a line bundle to $N$ is via its {\it canonical line bundle ($K_N=\det(T'N)$}, the determinant bundle of  the holomorphic tangent bundle $T'N$) and {\it the anti-canonical  line bundle ($K_N^{-1}$)}. The associated intrinsic curvature (i.e. $c_1(M)$, the Chern form of $K^{-1}_N$) is the {\it Ricci} curvature of $(N, h)$.  Compact K\"ahler manifolds with positive Ricci curvature  form a special class of  smooth projective/algebraic varieties, i.e. the {\it Fano manifolds}.  Its study and the extension to varieties with various singularities  have been one of active focuses of the algebraic geometry during  last decades. In this paper we study a family of  {\it intrinsic} curvature conditions (generalizing the Ricci curvature), whose positivity implies {\it the projectivity and   the  rational connectedness} of a compact K\"ahler manifold.

Rational connectedness is an important/useful property for algebraic manifolds \cite{Deb}. For compact K\"ahler manifolds with positive Ricci curvature this property was established by Campana \cite{Cam}, Koll\'ar-Miyaoka-Mori \cite{KMM}.  In this paper we show that

\begin{theorem}\label{thm:4} Let  $(N^n, h)$ be a compact K\"ahler manifold with $\Ric_k>0$, for some $1\le k\le n$. Then  $N$ is  projective and rationally connected. In particular, $\pi_1(N)=\{0\}$.
\end{theorem}

 The $\Ric_k$ is defined as the Ricci curvature of the $k$-dimensional holomorphic subspaces of the holomorphic tangent bundle $T'N$. Hence  it coincides with the holomorphic sectional curvature $H(X)$ when $k=1$, and with the Ricci curvature when $k=n=\dim_{\mathbb{C}}(N)$. The condition $\Ric_k>0$ is significantly different from its Riemannian analogue, i.e. the so-called $q$-Ricci (see next section for details), since it exams only the holomorphic subspaces in $T'N$, thus unlike its Riemannian analogue,  $\Ric_k>0$ does not imply $\Ric_{k+1}>0$. $\Ric_k>0$ means that every infinitesimal $k$-dimensional holomorphic subvariety is Fano. The notion of $\Ric_k$ was initiated in a recent study of the $k$-hyperbolicity of  a compact K\"ahler manifold by the author \cite{N}. It   is closely related to the degeneracy of holomorphic mappings from $\mathbb{C}^k$ into concerned manifolds (cf. Theorem 1.3 of \cite{N}). The condition $\Ric_k>0$ allows some  negativity of (holomorphic) sectional curvature if $k>1$. Note that all Hirzebruch surfaces (and generalized Hirzebruch manifolds)  admit K\"ahler metric with $\Ric_1>0$. This contracts sharply with the Fano condition of $\Ric>0$. The class of manifolds with $\Ric_k>0$ (for $k<n$) contains many non-Fano examples. It is interesting to find for what $k$ the class $\Ric_k>0$  has finite deformation connected components (cf. \cite{Nadel, KMM}  for the Fano case).

 The proof of Theorem \ref{thm:4} is completely different from that of  \cite{Cam, KMM, HeierWong}. Here it is built upon  recent techniques of  applying the  (partial) maximum principle  via the viscosity consideration developed  by the author in \cite{N, Ni-ijm}. It is in the proof of the projectivity,  Whitney's comass (amounts to an operator norm) is employed to localize the problem. The proof of rational connectedness also needs  a second variation consideration  of author (with F. Zheng \cite{Ni-Zheng2}) to obtain some desired estimates in the fiber direction of the flags in  $T'N$. The proof also uses a recent result  of \cite{CDP}.

The literature on  the fundamental group $\pi_1(N)$ of a K\"ahler manifold $N$ is big (cf. \cite{ABCKT}, \cite{Ni-ijm} and references there).
A result of Kobayashi \cite{Ko} asserts that a compact K\"ahler manifold with $\Ric>0$ must be simply-connected.  Same conclusion was proved  by Tsukamoto  \cite{Tsu} for compact K\"ahler manifold with $H>0$.   The next result of this paper provides an analogue of   Kobayashi's and Tsukamoto's theorems for K\"ahler manifolds with $\Ric^\perp>0$.
\begin{theorem}\label{thm:0-2} A compact K\"ahler manifold $N$ with  $\Ric^\perp>0$ must be simply-connected.
\end{theorem}
 The result was conjectured in \cite{NZ}.  In \cite{NZ},  motivated by the Laplace comparison theorem and the holomorphic Hessian comparison theorem, {\em orthogonal Ricci curvature}
$$ \Ric^{\perp}(X,\overline{X})\doteqdot\Ric(X, \overline{X}) - R(X, \overline{X}, X, \overline{X})/|X|^2, $$
for any type $(1,0)$ tangent vector $X$, was studied. Here $R(X,\overline{X}, X, \overline{X})$ is the {\it holomorphic sectional curvature}  (denoted as $H(X)$), which is the Gauss curvature of the infinitesimal curve tangent to $X$. For a compact K\"ahler manifold $N^n$ ($n=\dim_{\mathbb{C}}(N)$),  with $\Ric^{\perp}>0$ everywhere, its projectivity was shown in \cite{NZ}, via a uncommon unitary congruence normal form for $(2,0)$-forms.\footnote{Due to E. Cartan implicitly (L.-G. Hua explicitly in a paper of 1945), cf. p151 of \cite{NZ}. This partially contributes to  the gap between the first appearance of  $\Ric^\perp$  \cite{Mi-Pal} and meaningful results in \cite{NZ}, \cite{NWZ}, \cite{NZ-con}.} It was also proved  in \cite{NZ} that  $|\pi_1(N)|<\infty$ in general, and $\pi_1(N)=\{0\}$ for $n=2,3,4$.

   Unlike $\Ric$, $\Ric^\perp(X, \overline{X})$ does not come from a Hermitian symmetric sesquilinear form. But it can be viewed as the holomorphic sectional curvature of a Bochner curvature operator (namely the curvature operator which arises in the standard Bochner formula computing the Laplacian of the square of the norm of two forms, cf. \cite{NZ, NWZ}).  Despite this close connection with the holomorphic sectional curvature, the proof of Theorem \ref{thm:0-2}  follows the scheme of \cite{Ko} (for $\Ric>0$) via a Riemann-Roch-Hirzebruch formula and a vanishing theorem on Hodge numbers $h^{p, 0}$.  The vanishing of $h^{p, 0}$, $\forall 1\le p\le n$, needed in \cite{Ko} is known by Kodaira's vanishing theorem. However the proof of $h^{p, 0}=0, \forall 1\le p\le n$ under $\Ric^{\perp}>0$ (cf. Theorem  \ref{thm:1})  requires a completely new idea which involves a novel use of a viscosity consideration of Whitney's comass. The effective method also plays an important role in the proof of the above rational connectedness result. Note  that our proof applies to the case of $H>0$ (implying  Tsukamoto's result). It  provides a  unified argument for all cases of  $\{\Ric_k>0\}, k\in\{1,\cdots, n\}$, and $\Ric^\perp>0$   with additional information  $h^{p, 0}=0, \forall 1\le p\le n$.

 The study  of $\Ric^\perp>0$ in \cite{NZ-con, NWZ} is  also motivated by the so-called generalized Hartshorne conjecture (cf. Conjectures 11.1, 11.2 of \cite{CP} and Conjecture 8.23 of \cite{Zheng}): {\it A Fano manifold has nef tangent bundle if and only if it is a K\"ahler C-space.} The first curvature notion one naturally would like to associate with the nefness condition  is the so-called almost nonnegativity of bisectional  curvature. However, it has been proved  recently  that the almost nonnegativity of the  bisectional curvature \cite{BCW}    implies that the manifolds are  diffeomorphic to compact quotients of Hermitian symmetric spaces, provided the volume is noncollapsing.  A recent work of the author with X. Li \cite{LN} extends this to  the (weaker) almost nonnegativity of the orthogonal bisectional curvature.  In \cite{WYZ} a curvature positivity notion, namely  the  quadratic orthogonal bisectional curvature QB (cf. (\ref{eq:12}) for its definition)  was proposed (by Wu-Yau-Zheng) for the purpose of  a curvature characterization of the K\"ahler C-spaces.  Unfortunately  as shown in  \cite{ChauTam} it is  a bit off the target since only for about eighty percentage of classical K\"ahler C-spaces with  2nd Betti number $b_2=1$ (with the canonical K\"ahler-Einstein metric) have  QB$>0$, while for the rest twenty percent manifolds QB$<0$  somewhere. As a step back, the positivity of $\Ric^\perp$ was studied in \cite{NZ-con} for the purpose of the curvature characterization of C-spaces, since  on all classical C-spaces with $b_2=1$ the canonical K\"ahler-Einstein metrics satisfy $\Ric^\perp>0$ \cite{NWZ}. Further studies of compact K\"aher manifolds with $\Ric^\perp>0$ were carried in a recent work \cite{NWZ}: A complete classification for threefolds, a partial classification for fourfolds, and  a Frankel type result were obtained for compact K\"ahler manifolds with $\Ric^\perp>0$ in \cite{NWZ}. Many examples were also constructed in \cite{NZ, NWZ} illustrating that $\Ric^\perp$, $H$, and $\Ric$ are completely independent except the trivial relation $\Ric(X, \overline{X})=\Ric^\perp(X, \overline{X})+H(X)/|X|^2$.
  On the other hand, except for dimension $n=2,3$ (and $n=4$ if one is  optimistic) \cite{NWZ}, $\Ric^\perp>0$ appears  not enough to imply that the manifold is a K\"ahler C-space. This is partially reflected by  the flexible constructions of metrics with $\Ric^\perp>0$ on fiber bundles over a positively curved K\"ahler manifold in \cite{NWZ}.

 At the same time, motivated by the local rigidity theorem of Calabi-Vesentini \cite{CV}, the above relation between $\Ric^{\perp}$ and the generalized Hartshone conjecture,
   we introduce two stronger (than $\Ric^\perp$) notions of intrinsic curvature positivity, namely the {\it cross quadratic bisectional curvature} (abbreviated as CQB$>0$) and its dual $^d$CQB (cf. (\ref{eq:cqb-define}) and  (\ref{eq:dcqb})) in this paper.  The defining expressions appear similar  to the quadratic orthogonal bisectional curvature. However, a sharp  contrast is that the positivity of CQB and its dual can be verified   for all classical K\"ahler C-spaces with $b_2=1$ (cf. Theorems \ref{thm:NWZ1}, \ref{thm:NWZ12}). Results as initial studies of these two notions of curvature in this paper includes:  (1)  CQB$>0$ implies $\Ric^\perp>0$; (2) The projectivity and simply-connectedness of manifolds with  CQB$>0$ or $^d$CQB$>0$ (cf. Theorem \ref{thm:3}); (3) A deformation rigidity result for manifolds with  quasi-positive $^d$CQB (or quasi negative CQB).  Since there are non locally Hermitian symmetric manifolds with  CQB$<0$ (3) generalizes the result of Calabi-Vesentini \cite{CV}. Utilizing the K\"ahler-Ricci flow,   the Fanoness was proved under the assumptions of CQB$\ge 0$ (or $^d$CQB$\ge 0$) and the  finiteness of $\pi_1(N)$  recently in \cite{NZ-cross} joint with F. Zheng.  In particular CQB$> 0$ (or $^d$CQB$> 0$) implies that $N$ is Fano. Hence there is a good chance that one of these two  curvature notions can provide the curvature characterization of the K\"ahler C-spaces. Tracing  $^d$CQB (CQB) leads to a related notion of Ricci curvature,  namely $\Ric^+$ ($\Ric^\perp$ respectively). We also show that a compact K\"ahler manifold with $\Ric^+>0$ is projective and simply-connected. It is also proved in this paper that any compact K\"ahler manifolds with quasi-positive $\Ric^\perp$ and of Picard number one must be Fano. However, the rational connectedness of manifolds with $\Ric^{\perp}>0$ remains unknown. Since the condition $\Ric^\perp>0$ allows  arbitrarily large $b_2$ (cf. \cite{NZ-cross} for examples of Type A K\"ahler C-spaces  with CQB$\ge 0$, $\Ric^{\perp}>0$, $^d$CQB$>0$, and   with arbitrarily large $b_2$), the implications of $\Ric^\perp>0$ on the dimension of certain harmonic $(1,1)$-forms is included in the appendix.

 The well-known curvature notions for K\"ahler manifolds   include the {\it sectional curvature}, the  {\it bisectional curvature} $B(X, Y)\doteqdot R(X, \overline{X}, Y, \overline{Y})$, and the {\it  holomorphic sectional curvature}, $H(X)$ mentioned above.
  For Hermitian manifolds there are Griffiths' positivity \cite{Gri}. Restricted to K\"ahler manifolds, it is the same as $B>0$. Various positivity notions in algebraic geometry are discussed in the excellent books of Lazarsfeld \cite{La12}. However the positivity (even the nonnegativity) of bisectional curvature is rather restrictive for compact K\"ahler manifolds since Mori's solution of the Hartshorne conjecture \cite{Mori} asserts  that if $T'N$ is ample  the complex manifold $N=\mathbb{P}^n$, the complex projective space. In particular, since $B>0$ implies that  $T'N$ is ample (cf. Theorem 6.1.25 of \cite{La12}), Mori's result implies that the only compact K\"ahler manifold with $B>0$ is  $\mathbb{P}^n$ (cf.  \cite{SY}, for an independent  K\"ahler  geometric proof of Siu-Yau).

Related to results of this paper, the projectivity of compact K\"ahler manifolds with $S_2>0$ was recently proved in \cite{Ni-Zheng2}, generalizing an earlier result of \cite{XYang} under the stronger assumption $H>0$. A general (maybe the most general possible) result on a weaker criterion of the projectivity can be found in Corollary \ref{coro:6-add-3} (i.e. any K\"ahler manifold of BC-$2$ positive curvature is projective). The vanishing  of $h^{p, 0}$ in Theorem \ref{thm:4} and \ref{thm:0-2} is also extended to broader cases in  Theorem \ref{thm:6-added-ricciperp}.  For  algebraic manifolds with $H>0$ the rational connectedness was  proved in \cite{HeierWong} (cf. also \cite{Yang18}). Recent work \cite{Ni-ijm} also contains results on the fundamental groups of compact K\"ahler manifolds with $\Ric_k\ge 0$, in particular $H\ge 0$.
  We hope that this paper serves an introduction to relatively new notions of positivity concerning the intrinsic metrics, namely $\Ric_k$, $\Ric^\perp_k$, CQB$_k$, as well as their dual $\Ric^+_k$, $^d$CQB$_k$,  for K\"ahler manifolds. One can find many questions/open problems, and examples,  in later sections of this paper.

\section{Definitions and statements of results}

We start  with the following conjecture  proposed in  \cite{NZ} (cf. Conjecture 1.6).

\begin{conjecture}\label{conj} Let $N^n$ ($n\geq 2$) be a compact K\"ahler manifold with $\Ric^{\perp} >0$ everywhere. Then for any $1\leq p\leq n$, there is no non-trivial global holomorphic $p$-form, i.e. the Hodge number $h^{p,0}=0$. In particular, $N^n$ is simply-connected.
\end{conjecture}

The conjecture was confirmed  for $n=2, 3, 4$ in \cite{NZ} following a general scheme of Kobayashi. As illustrated in \cite{NZ}, the  ``in particular" part, namely the simply-connectedness of compact K\"ahler manifolds,  would follow from Hirzebruch's Riemann-Roch formula \cite{Hir} as follows: Letting ${\mathcal O}_N$ be the structure sheaf,  the Euler characteristic number
$$ \chi({\mathcal O}_N) \doteqdot 1 - h^{1,0} + h^{2,0} - \cdots + (-1)^n h^{n,0}$$
satisfies that $\chi({\mathcal O}_{\widetilde{N}})=\nu \cdot \chi({\mathcal O}_N)$ by the Riemann-Roch-Hirzebruch formula,  if $\widetilde{N}$ is a finite  $\nu$-sheets covering of $N$. On the other hand, the vanishing of all Hodge numbers $h^{p, 0}$ for $1\le p\le n$ (which is the main part of the conjecture) asserts that $\chi({\mathcal O}_N) =1$ for both $N$ and $\widetilde{N}$, if $\widetilde{N}$ is compact and of $\Ric^\perp>0$ (hence projective). This  forces $\nu=1$, hence $\pi_1(N)=\{0\}$. Note that the universal cover $\widetilde{N}$ satisfies  $\Ric^\perp>\delta>0$. Hence $\widetilde{N}$ is  compact and projective by Theorem 3.2 of \cite{NZ}. This argument was the one used in \cite{Ko, NZ} proving the simply-connectedness of a Fano manifold, and for $n=2,3,4$ with $\Ric^\perp>0$.

In this paper we  prove Conjecture \ref{conj} for all $n\ge2$ by  a stronger result, namely  the  vanishing of $h^{p, 0}$ under a weaker curvature condition related to $p$. First we recall this condition (cf.  Section 4 of \cite{NZ}). Let $\Sigma$ be a  $k$-subspace $\Sigma\subset T_x'N$.
Let $\aint f (Z)\, d\theta(Z)$ denote $\frac{1}{Vol(\mathbb{S}^{2k-1})}\int_{\mathbb{S}^{2k-1}}f(Z)\, d\theta(Z)$, where $\mathbb{S}^{2k-1}$ is the unit sphere in $\Sigma$.  Define
\begin{equation}\label{eq:11}
S^{\perp}_k(x, \Sigma)\doteqdot k \aint_{Z\in \Sigma, |Z|=1} \Ric^\perp(Z, \overline{Z})\, d\theta(Z)
\end{equation}
Similarly  $S_k(x, \Sigma)$, the $k$-scalar curvature of $\Sigma$, can be  defined by  replacing $\Ric^\perp(Z, \overline{Z})$ with $\Ric(Z, \overline{Z})$ in (\ref{eq:11}). Let $S_k^{\perp}(x)\doteqdot \inf_{\Sigma} S^{\perp}_k(x, \Sigma)$.
Thus  $S_k^{\perp}(x)>0$ if and only if  for any $k$-subspace $\Sigma\subset T_x'M$, $S^{\perp}_k(x, \Sigma)>0$. The condition  $S_k^\perp(x)>0, k\in \{1, \cdots, n\}$, interpolate between $\Ric^\perp(X, \overline{X})$ and $\frac{n-1}{n+1}S(x)$ (see Lemma \ref{lm:Gold}). It is easy to see that $S_\ell^\perp>0$ implies $S^\perp_k>0$ for $k\ge \ell$. And it is not hard to prove that (cf. (\ref{eq:24}))
$$S^{\perp}_k(x, \Sigma)=\left(\Ric(E_1, \overline{E}_1)+\Ric(E_2, \overline{E}_2)+\cdots +\Ric(E_k, \overline{E}_k) \right)-\frac{2}{(k+1)}S_k(x, \Sigma).$$
 The corresponding collection of $k$-scalar curvatures $\{S_k(x),  k=1, \cdots, n\}$ interpolates between the holomorphic sectional curvature $H(X)$ and the scalar curvature $S(x)$.
The equation (\ref{eq:11}) in particular implies that $S^\perp_n(x)=\frac{n-1}{n+1}S(x).$ The first  theorem of this paper proves that $h^{p, 0}=0$, $\forall p\ge k$, if $S_k^{\perp}>0$ or $S_k>0$. Conjecture \ref{conj} follows since $\Ric^\perp>0$ implies that $S_k^\perp>0$, $\forall k \in \{1, \cdots, n\}$.

\begin{theorem}\label{thm:1}
Let $(N, g)$ be a compact K\"ahler manifolds such that $S^{\perp}_k(x)>0$ for any $x\in N$. Then
$h^{p, 0}=0$ for any $p\ge k$. The same result holds if $S_k>0$. In particular, if $\Ric^\perp>0$ (or $H>0$), then $h^{p, 0}(N)=0$ for all $1\le p\le n$, and $N$ is simply-connected.
\end{theorem}

The  part $h^{2, 0}=0$ was proved in \cite{NZ, Ni-Zheng2} under the assumption $S_2(x)>0$.  The argument there is limited to $p=2$ since one has to use a normal form for $(2,0)$-forms. The proof here uses  a different  idea which is developed recently in  \cite{N} to  prove a new Schwarz Lemma by the author. We recall that idea first before explaining the related details. Starting from the work of Ahlfors, the Schwarz Lemma concerns estimating the gradient of a holomorphic map $f$ between two K\"ahler (or Hermitian) manifolds $(M^m, h)$ and $(N^n, g)$. For that it is instrumental to study the pull-back $(1,1)$-form $f^*\omega_g$, where $\omega_g$ is the K\"ahler form of $(N, g)$. The traditional approach (before the work of \cite{N}) is to compute the Laplacian of the trace of $f^*\omega_g$. But in \cite{N}, the author estimated the largest singular value of $df$, equivalently the biggest eigenvalue of $f^*\omega_g$, by applying the $\partial\bar{\partial}$-operator  to the maximum eigenvalue of $f^*\omega_g$ (which is only continuous in general) via a viscosity consideration. It allows the author to prove another {\it natural} generalization of Ahlfors' result with a sharp estimate on the largest singular value of $df$ in terms of the holomorphic sectional curvatures of both the domain and target manifolds. This estimate can be viewed as a complex version of Pogorelov's estimate for solutions of the Monge-Amp\`ere equation \cite{Po}. To prove the vanishing of holomorphic $(p, 0)$-forms under the assumption of $\Ric^\perp>0$, in Section 3 we apply the $\partial \bar{\partial}$-operator  on the {\it comass} of holomorphic $(p, 0)$-forms (cf. \cite{Fede, Whit}), through a similar viscosity consideration. The comass of a $(p, 0)$ form generalizes  the biggest singular value of $df$ in some sense since one can view $df$ as a vector valued $(1, 0)$-form. It can be done thanks to some basic properties of the comass established by  Whitney \cite{Whit}. This new idea also allows us to prove a generalization of the main theorem in \cite{Ni-Zheng2} (cf.   Corollary \ref{coro:6-add-3}).

  By combining this new idea with the work of \cite{Ni-Zheng2}, in Section 4, we prove the projectivity and the rational connectedness of compact K\"ahler  manifolds under the condition  $\Ric_k>0$, i.e. Theorem \ref{thm:4}. The notion $\Ric_k$ was introduced in \cite{N} to prove that {\it any  K\"ahler manifold with $\Ric_k<0$ uniformly must be $k$-hyperbolic}, a concept generalizing the Kobayashi hyperbolicity (which amounts to $1$-hyperbolic).  Let $\Ric_k(x)\doteqdot \inf_{\Sigma, v\in \Sigma, |v|=1} \Ric_k(x, \Sigma) (v, \bar{v})$. Here $\Ric_k(x, \Sigma)$ is the Ricci curvature of the curvature tensor $R$ restricted to $\Sigma\subset T_xN$.
  Thus $\Ric_k(x)>\lambda(x)$  if and only if $\Ric_k(x, \Sigma)(v, \bar{v})>\lambda |v|^2$, for any $v\in \Sigma$ and  for every $k$-dimensional subspace $\Sigma$.  We say $N$ has positive $\Ric_k$ if $\Ric_k(x)>0, \forall x\in N$. The condition $\Ric_1>0$  is equivalent to that the holomorphic sectional curvature $H>0$. For $k=n$, $\Ric_k$ is  the Ricci curvature. By \cite{Hitchin, AHZ} $\Ric_k>0$ is independent from $\Ric_\ell>0$ for $k\ne \ell$ (cf. also \cite{YZ, NZ} for more examples). The known examples of manifolds with $\Ric_k>0$ for $k\ne n$ contain mostly {\it non-Fano} manifolds. It remains interesting to find out for what $k$ and $n$ the deformation types of manifolds with $\Ric_k>0$  is finite.

   As in \cite{NZ,Ni-Zheng2} the projectivity only needs $h^{2, 0}=0$.   In  Theorem \ref{thm:61-add} we show  a stronger vanishing result: {\it $h^{p,0}=0$ for any $1\le p\le n$ under the assumption $\Ric_k>0$ for some $1\le k\le n$}. (In Theorem \ref{thm:6-added-ricciperp} this result is extended to $\Ric^{\perp}_k>0$ and $\Ric^+_k>0$). The rational connectedness is proved in Section 4 by  another vanishing theorem, whose validity is  a criterion of the rational connectedness, thanks to \cite{CDP}. Both the second variation estimate from \cite{Ni-Zheng2} and the one utilizing the comass for $(p, 0)$-forms introduced in Section 3 of this paper are crucial in proving these two vanishing theorems. Theorem \ref{thm:4} generalizes both  the result for  Fano manifolds \cite{Cam, KMM} (the case $k=n$, namely the Fano case of Campana, Koll\'ar-Miyaoka-Mori), and the more recent result for the compact K\"ahler  manifolds with positive holomorphic sectional curvature \cite{HeierWong} by Heier-Wong (cf. also \cite{XYang} for the projectivity for the case $k=1$), since $\Ric_1>0$ amounts to $H>0$ and $\Ric_n=\Ric$. It is not clear if  $\Ric_k>0$ has anything to do with that Ricci curvature is $k$-positive in general. When $k=1$, Hitchin's examples show that they are independent. However $\Ric_k$ is related to the notion of  $q$-Ricci studied in Riemannian geometry which interpolates the Ricci and the sectional curvature. In particular, it is the complex analogue of $q$-Ricci and if the  $(2k-1)$-Ricci is positive in the sense of Bishop-Wu \cite{Bishop, Wu} then  $\Ric_k>0$. The positivity of the $(2k-1)$-Ricci is a much stronger condition than $\Ric_k>0$ since it  requires the Ricci being positive  on all $2k$-dimensional subspaces of the (complexified) tangent space $T_xN$  $\forall x\in N$. This makes that the positivity of the $p$-Ricci implies the positivity of $q$-Ricci if $q\ge p$. On the other hand since most of $2k$-dimensional (real) subspaces of $T_xN\otimes \mathbb{C}$ are  neither invariant under the almost complex structure, nor  subspaces of $T'N$, $\Ric_k>0$ is a lot weaker than $(2k-1)$-Ricci being positive. A major difference is that $\Ric_k> 0$ does not imply $\Ric_{k+1}> 0$, unlike the $q$-Ricci positivity condition.

In Section 5 of the paper we study the question when a compact K\"ahler manifold with $\Ric^\perp>0$ is Fano, a question raised in \cite{NZ}. We give an affirmative answer under an extra assumption.

\begin{theorem} \label{thm:c1} Let $(N, h)$ be a compact K\"ahler manifold of complex dimension $n$.
 Then  (i) if $\Ric^\perp$ is quasi-positive (namely $\Ric^{\perp}\ge 0$ everywhere and $\Ric^{\perp}>0$ somewhere)  and the Picard number $\rho(N)=1$, then $N$ must be Fano;  (ii) if $\Ric^\perp$ is quasi-negative and $h^{1,1}(N)=1$, $N$ must be projective with ample canonical line bundle $K_N$. In particular in the case (i) $N$ admits a K\"ahler metric with positive Ricci, and in the case of (ii) $N$ admits a K\"ahler-Einstein metric with negative Einstein constant.
\end{theorem}

Since it was proved in \cite{NZ} that $N$ is projective and $h^{1, 0}(N)=h^{2, 0}(N)=0=h^{0, 2}(N)=h^{0,1}(N)$ under the assumption that $\Ric^\perp>0$, the assumption of $\rho(N)=1$ for case (i) is equivalent to the assumption that the second Betti number $b_2=1$. We should mention that in \cite{NWZ}, it has been shown that for all K\"ahler $C$-spaces of classical type with $b_2=1$ the canonical K\"ahler-Einstein metric satisfies $\Ric^\perp>0$.

To put Theorem \ref{thm:c1}  into  perspectives  it is appropriate to recall some  earlier works. First related to $\Ric^\perp\ge 0$  there exists a stronger condition called the {\it nonnegative quadratic orthogonal bisectional sectional curvature},   studied by various people including authors of \cite{WYZ} and  \cite{CT},  etc. The
{\it quadratic orthogonal bisectional curvature}  (abbreviated as QB), is defined for any real vector $\vec{{\bf a}}=(a_1, \cdots, a_n)^{tr}$ and any unitary
 frame $\{E_i\}$ of $T'_xN$,
QB$(\vec{{\bf a}})=\sum_{i, j}R_{i\bar{i}j\bar{j}}(a_i-a_j)^2$.  In \cite{NT2} it was  formulated invariantly as a quadratic form  on the space of Hermitian symmetric tensors. Precisely for symmetric tensor $A$,
$
QB_R(A)\doteqdot \langle R, A^2\bar{\wedge}id -A\bar{\wedge}A\rangle.
$
Interested readers can refer to \cite{NT2} for the notations involved.  For   any  unitary orthogonal frame of $T'N$   $\{E_\alpha\}$ we adapt
\begin{equation}\label{eq:12}
QB_R(A)\doteqdot \sum_{\alpha, \beta=1}^nR(A(E_\alpha), \overline{A(E_\alpha)}, E_\beta, \overline{E}_\beta)-R(E_\alpha, \overline{E}_\beta, A(E_\beta), \overline{A(E_\alpha)}).
\end{equation}
  Clearly it is independent of the choice of the unitary frame.
Its nonnegativity, abbreviated as NQOB, is equivalent to that QB$(\vec{{\bf a}})\geq 0$ for any $\vec{{\bf a}}$ with respect to any unitary frame $\{ E_i\}$. NQOB was formally introduced in \cite{WYZ} (appeared implicitly in the work of Bishop-Goldberg in 1960s). It is easy to see that QB$>0$ implies $\Ric^{\perp}>0$.\footnote{Motivated by the work of Calabi-Vesentini \cite{CV},   we  introduce the so-called {\it cross quadratic bisectional curvature} (abbreviated as CQB), another (quadratic form type) curvature, whose positivity also implies $\Ric^\perp>0$.} In \cite{CT}  the following was proved  Chau-Tam in \cite{CT}, Theorem 4.1:

\begin{theorem}[Chau-Tam]\label{thm:CT} Let $(N, h)$ be a compact K\"ahler manifold with NQOB and  $h^{1,1}(N)=1$. Assume further that  $N$ is locally irreducible then $c_1(M)>0$.
\end{theorem}

Theorem \ref{thm:c1} has the following corollary, which extends the above result.

\begin{corollary}\label{coro:11} Let $(N, h)$ be a compact K\"ahler manifold of complex dimension $n$ with $\Ric^{\perp}\ge 0$.
 Assume further that $h^{1, 1}(N)=1$ and $N$ is locally irreducible.  Then $c_1(N)>0$, namely $N$ is Fano. A similar result holds under the assumption $\Ric^\perp \le 0$.
\end{corollary}

 There exists compact K\"ahler manifolds with $b_2>1$ (cf. construction in \cite{NWZ} via projectivized bundles) and $\Ric^\perp>0$. Hence it remains an interesting question whether or not  the same conclusion of Theorem \ref{thm:c1} (i)  holds without the assumption $h^{1,1}=1$. Since the quasi-positivity of  QB implies that $h^{1,1}=1$, as a consequence we have that {\it any compact K\"ahler manifold with quasipositive QB must be Fano}.
  Whether or not the same conclusion of part (ii) of Theorem \ref{thm:c1}  holds without assuming that $h^{1, 1}=1$ remains open.

As mentioned in Section 1, motivated by the relation of the condition $\Ric^\perp>0$  with the generalized Hartshone conjecture and the work of Calabi-Vesentini we introduce the {\it cross quadratic bisectional curvature} (abbreviated as  CQB)  as a Hermitian quadratic form on the space of linear maps $A: T''_xN \to T'_xN$:
\begin{equation}\label{eq:cqb-define}
CQB_R(A)\doteqdot \sum_{\alpha, \beta=1}^nR(A(\overline{E}_\alpha), \overline{A(\overline{E}_\alpha)}, E_\beta, \overline{E}_\beta)-R(E_\alpha, \overline{E}_\beta, A(\overline{E}_\alpha), \overline{A(\overline{E}_\beta)}).
\end{equation}
Here   $\{E_\alpha\}$ is a unitary frame of $T'_xN$. It is easy to see that CQB$(A)$ is independent of the choice of the unitary frame. Dually, the {\em dual cross quadratic bisectional curvature} ($^d$CQB) is defined as a Hermitian quadratic form on linear maps $A: T'N \rightarrow T''N $:
\begin{equation}\label{eq:dcqb}
 ^d\mbox{CQB}_R(A) \doteqdot \sum_{\alpha , \beta =1}^n  R(\overline{A(E_{\alpha})}, A(E_{\alpha} ), E_{\beta} , \overline{E}_{\beta} ) + R(E_{\alpha} , \overline{E}_{\beta}, \overline{  A(E_{\alpha} )} , A(E_{\beta} ) ). \end{equation}

The advantage of these two curvature notions over QB is demonstrated by

\begin{theorem}\label{thm:dr-cv1} (i) Let $N^n$ be a classical K\"ahler C-space with $n\geq 2$ and $b_2=1$ Then  the canonical metric satisfies CQB$>0$ and $^d$CQB$>0$.

(ii) For a compact K\"ahler manifolds with quasi-positive $^d$CQB (or quasi-negative CQB),
$H^q(N, T'N)=\{0\}$, for $1\le q\le n$, and $N$ is deformation rigid in the sense that it does not admit nontrivial infinitesimal holomorphic deformation. In particular the deformation rigidity holds for all classical K\"ahler C-spaces with $b_2=1$.
\end{theorem}
The proof uses the results  of \cite{CV}, \cite{ChauTam}, \cite{Itoh}, and \cite{NWZ}.  We also use  the criterion of Fr\"olicher and Nijenhuis \cite{FN, KS} for the deformation rigidity statement.
The key is a Kodaira-Bochner formula (cf. \cite{CV})  and the role of a curvature notion $^d$CQB ({\it dual-cross quadratic bisectional curvature}) played in the Kodaira-Bochner formulae.
The rigidity result on K\"ahler C-spaces  can possibly  be implied by a result of Bott \cite{Bott}. Here it follows from a general vanishing theorem for manifolds with $^d$CQB$>0$.
Hence  before a complete classification of K\"aher manifolds with $^d$CQB$>0$ (cf. \cite{NZ-cross} for a precise conjecture related to this) the rigidity result above for $N$ with $^d$CQB$>0$ is a more general statement.  One can refer Sections 5 and 6 for further motivations and detailed discussions on these two new curvatures.

The new dual-cross quadratic bisectional curvature $^d$CQB  naturally induces a Ricci type curvature (in a  similar manner  as QB and  CQB induces $\Ric^\perp$.)  It is denoted by $\Ric^+$,  and  is defined,   for any $X\in T'_xN$, as
$$
\Ric^{+}(X,\overline{X})=\Ric(X, \overline{X})+H(X)/|X|^2.
$$
In Section 6, for the K\"ahler manifolds with $\Ric^+>0$ we have the following result similar to the $\Ric^\perp>0$ case.

\begin{theorem}\label{thm:3} Let $(N, h)$ be a complete K\"ahler manifold with $\Ric^+\ge \delta>0$ (or replaced with any one of  $\{\Ric_k\ge \delta, \Ric^{\perp}_k\ge \delta, \Ric^{+}_k\ge \delta\}$). Then (i) $N$ is compact, (ii)
$h^{p, 0}=0$ for all $n\ge p\ge 1$. In particular, $N$ is simply-connected and $N$ is projective. Since $^d$CQB$>0$ implies $\Ric^+>0$, this applies to compact manifolds with $^d$CQB$>0$.
\end{theorem}

The proof of the above result again makes use of the method via a viscosity consideration on  the comass  (introduced in Section 3) and follows a similar line of argument as the proof of Theorem \ref{thm:1}. In Section 6 we also prove a diameter estimate and a result similar to Corollary \ref{coro:11} for $\Ric^+$.

The {\it cross quadratic bisectional curvature}  and its dual $^d$CQB  are shown   positive on some {\it exceptional K\"ahler C-spaces} too. Since CQB$>0$ (as QB$>0$) implies $\Ric^\perp>0$, Theorem \ref{thm:dr-cv1} generalizes the result of \cite{NWZ}. On the other hand it was shown by Chau-Tam \cite{ChauTam} that QB$>0$ fails to hold for all K\"ahler C-spaces with $b_2=1$,  and it was shown in \cite{NWZ} that there exists a non-homogenous  compact K\"ahler manifold with $\Ric^\perp>0$.  Hence  one of these two new curvature notions will more likely give  a curvature characterization of the compact K\"ahler C-spaces with $b_2=1$. Towards this direction we prove (in Theorem \ref{thm:cqb-rc}) that {\it  a compact K\"ahler manifold with CQB$>0$ must be rationally connected.} This also follows from \cite{Cam, KMM} and  the statement that   CQB$>0$ (or $^d$CQB$>0$) implies that $N$ is Fano (cf. \cite{NZ-cross}). More ambitious project is to apply these curvatures to tackle  the generalized Hartshorne conjecture concerning the Fano manifolds with a nef tangent bundle (cf. conjectures formulated in \cite{NZ-cross}). We also calibrate QB, CQB and $^d$CQB into  QB$_k$, CQB$_k$ or $^d$CQB$_k$ with $k\in \{1, \cdots, n\}$ to bridge them with $\Ric^\perp$ and $\Ric^+$.

In  the appendix, we study the gap in terms of vanishing theorems between QB$>0$ and $\Ric^\perp>0$.
Most results in this paper can be adapted to Hermitian manifolds without much difficulty, if the notions of involved curvatures are properly extended.

\section{Comass and the proof of Theorem \ref{thm:1}}
In \cite{N} and \cite{Ni-ijm} we developed a viscosity technique to apply a maximum principle to the operator norm of the differential of a holomorphic map. Here we extend the idea to differential forms. The {\it comass}  introduced by Whitney fits our need  quite well. We start with a brief summary of its properties.
Let $V$ be a Euclidean space. A $r$-(multi) vector ${\bf a}$ is an element of $\wedge_r V$, namely the space of $r$-multi linear skew symmetric forms  on $V^*$ (the dual of $V$). Here we identify $V$ and $V^*$ via the inner product when needed. A vector ${\bf a}$  is called simple if there exists $v_1, \cdots, v_r\in V$ such that ${\bf a}=v_1\wedge \cdots \wedge v_r$. This can be defined for $r$-covector $\omega$ similarly. For a $r$-covector $\omega$ the comass is defined in \cite{Whit} as
$$
\|\omega\|_0\doteqdot \sup \{ |\omega ({\bf a})|: {\bf a} \mbox{ is a simple $r$-vector}, \|{\bf a}\|=1\}.
$$
Here the norm $\|\cdot\|$ is the norm (an $L^2$-norm in some sense) induced by the inner product defined for simple vectors ${\bf a}=x_1\wedge \cdots \wedge x_r, {\bf b}=y_1\wedge \cdots \wedge y_r$, with $x_i, y_j \in V$, as
$$
\langle {\bf a}, {\bf b}\rangle \doteqdot \det(\langle x_i, y_j\rangle )
$$
and then extended bi-linearly to all $r$-covectors ${\bf a}$ and ${\bf b}$ which are linear combination of simple vectors.
The following results concerning the comass are well-known. The interested readers can find their proof in Whitney's classics \cite{Whit} (p52-55, Theorem 13A, Lemma 13a) or Federer's \cite{Fede} (Section 1.8).

\begin{theorem}[Whitney] \label{prop:21} (i) $\|\omega\|_0$ is a normal and $\|\omega\|_0=\sup \{|\omega({\bf a})|: \|{\bf a}\|_0=1\}$, where $\|{\bf a}\|_0$ is the mass of ${\bf a}$  defined as
$$
\|{\bf a}\|_0\doteqdot \inf \{\sum \|{\bf a}_i\|: {\bf a}=\sum {\bf a}_i, \mbox{ the } {\bf a}_i \mbox{ simple}\}.
$$

(ii) For any r-vector {\bf a}, $\|{\bf a}\|_0\ge \|{\bf a}\|$, with equality if and only if ${\bf a}$ is simple.

(iii) For each $\omega$ there exists a $r$-vector ${{\bf b}}$ such that $\|\omega\|_0 =|\omega({\bf b})|$, ${\bf b}$ is simple, and $\|{\bf b}\|=1$.

(iv) If $\omega$ is simple, $\|\omega\|_0=\|\omega\|$.

(v) $ \|\omega\|\ge \|\omega\|_0\ge \left(\frac{r! (n-r)!}{n!}\right)^{\frac{1}{2}} \|\omega\|$, with the first inequality holds equality if and only if $\omega$ is simple.
\end{theorem}

We shall prove the theorem via an argument by contradiction. Assume that $S_k^{\perp}>0$ and there exists a $\phi\ne 0$ which  is a harmonic $(p, 0)$-form with $p\ge k$. It is well known that it is holomorphic. Let $\|\phi\|_0(x)$ be its comass at $x$. Then its maximum (nonzero) must be attained somewhere at $x_0\in N$. We shall exam $\phi$ more closely in a coordinate chart (to be specified later) of $x_0$.  By the above proposition, at $x_0$, there exits a simple  $p$-vector ${\bf b}$ with $\|{\bf b}\|=1$, which we may assume to be $\frac{\partial}{\partial z_1}\wedge \cdots \wedge \frac{\partial}{\partial z_p}$ for a unitary  frame  $\{\frac{\partial}{\partial z_k}\}_{k=1, \cdots, n}$ at $x_0$, such that $\max_{x\in N} \|\phi\|_{0}(x)=\|\phi\|_0(x_0)=|\phi({\bf b})|$. If we denote $\phi=\frac{1}{p!}\sum_{I_p} a_{I_p}dz^{i_1}\wedge \cdots \wedge dz^{i_p}$,  where $I_p=(i_1, \cdots, i_p)$ runs all $p$-tuples with $i_s\ne i_t$ if $s\ne t$, we deduce
$$
\|\phi\|_0(x_0)=|a_{12\cdots p}|(x_0).
$$
Extend the frame to have a normal complex coordinate chart $U$ centered at $x_0$. This means that at $x_0$, the metric tensor $g_{\alpha\bar{\beta}}$ satisfies (cf, \cite{Tian})
$$
g_{\alpha \bar{\beta}}=\delta_{\alpha \beta}, \quad dg_{\alpha \bar{\beta}} =0,\quad  \frac{\partial^2 g_{\alpha\bar{\beta}}}{\partial z_{\gamma}\partial z_{\delta}}=0.
$$
With respect to this coordinate  $\phi(x)=\frac{1}{p!}\sum_{I_p} a_{I_p}(x)dz^{i_1}\wedge \cdots \wedge dz^{i_p}$ for $x\in U$ with $a_{I_p}(x)$ being holomorphic. Let $$\widetilde{\phi}(x)\doteqdot  \frac{1}{\sqrt{\det(g_{\alpha\bar{\beta}})_{1\le \alpha, \beta\le p} } \sqrt{\det(g^{\alpha\bar{\beta}})_{1\le \alpha, \beta\le p}}} a_{12\cdots p}(x)  dz^{1}\wedge \cdots \wedge dz^{p}.$$
This is defined in $U$. Since $\widetilde{\phi}$ is simple
$\|\widetilde{\phi}\|_0^2(x)=\|\widetilde{\phi}\|^2(x) = \frac{|a_{1\cdots p}|^2(x)}{\det(g_{\alpha\bar{\beta}})_{1\le \alpha, \beta\le p}}$.
In particular $\|\widetilde{\phi}\|_0(x_0)=\|\phi\|_0(x_0)$. On the other hand let ${\bf a}=\frac{\partial}{\partial z^1}\wedge \cdots \wedge \frac{\partial}{\partial z^p}$.
Then by the definition of the comass $\|\cdot\|_0$, namely  $\|\phi\|_0=\sup \frac{|\phi({\bf a})|}{\|{\bf a}\|}$, taking among all simple nonzero ${\bf a}$, 
$$
\|\phi\|^2_0(x)\ge \frac{|\phi({\bf a})|^2}{\|{\bf a}\|^2}=\frac{|a_{1\cdots p}|^2(x)}{\det(g_{\alpha\bar{\beta}})_{1\le \alpha, \beta\le p}}\ge \|\widetilde{\phi}\|_0^2(x).
$$  
Since $\|\phi\|_0(x)\le \|\phi\|_0(x_0)$, as a consequence we  have that
$$
\|\widetilde{\phi}\|(x)=\|\widetilde{\phi}\|_0(x)\le \|\phi\|_0(x)\le \|\phi\|_0(x_0)=|a_{1\cdots p}(x_0)|=\|\widetilde{\phi}\|_0(x_0)=\|\widetilde{\phi}\|(x_0).
$$
In summary, we have constructed a simple  $(p, 0)$-form $\widetilde{\phi}(x)$ in the neighborhood of $x_0$ such that its $L^2$-norm attains its maximum value at $x_0$.

Now apply $\partial_v \partial_{\bar{v}}$ to $\log \|\widetilde{\phi}\|^2$ at $x_0$.  If $v=\frac{\partial}{\partial z^\gamma}$ we have that at point $x_0$ that
\begin{equation}\label{eq:22}
0\ge  -\sum_{\alpha, \beta=1}^p g^{\alpha \bar{\beta}} \frac{\partial^2}{\partial z^{\gamma}\partial \bar{z}^\gamma} g_{\alpha\bar{\beta}}=\sum_{\alpha=1}^p R_{\alpha \bar{\alpha} \gamma \bar{\gamma}}.
\end{equation}
Namely we have arrived that at $x_0$, 
\begin{equation}\label{eq:23}
0\ge \sum_{j=1}^p R_{v\bar{v}j\bar{j}}.
\end{equation}
Now we are essentially at the same position of the proof in \cite{NZ}. For the sake of the completeness we include the argument below. Let $\Sigma=\mbox{span}\{\frac{\partial}{\partial z_1}, \cdots, \frac{\partial}{\partial z_p}\}$. It is easy to see from (\ref{eq:23}) that $S_p(x_0, \Sigma)\le 0$, where $S_p(x_0, \Sigma)$ denotes the scalar curvature of the curvature $R$ restricted to $\Sigma$. In fact
$S_p(x_0, \Sigma)=\sum_{i, j=1}^p R_{i\bar{i} j\bar{j}}$.

On the other hand as in \cite{NZ}
\begin{eqnarray}
\frac{1}{p} S^{\perp}_p(x_0, \Sigma)&=&\aint_{Z\in \Sigma, |Z|=1} \Ric^\perp(Z, \overline{Z})\, d\theta(Z)=\aint_{Z\in \Sigma, |Z|=1} \left(\Ric(Z, \overline{Z})-H(Z)\right)\, d\theta(Z)\nonumber\\
&=&\aint \frac{1}{Vol(\mathbb{S}^{2n-1})}\left(\int_{\mathbb{S}^{2n-1}} \left(nR(Z, \overline{Z}, W, \overline{W})-H(Z)\right)\, d\theta(W)\right)\, d\theta(Z)\nonumber\\
&=&\frac{1}{Vol(\mathbb{S}^{2n-1})}\int_{\mathbb{S}^{2n-1}}\left( \aint \left(nR(Z, \overline{Z}, W, \overline{W})-H(Z)\right)\, d\theta(Z)\right)d\theta(W)\nonumber\\
&=& \frac{1}{p}\left(\Ric_{1\bar{1}}+\Ric_{2\bar{2}}+\cdots +\Ric_{p\bar{p}} \right)-\frac{2}{p(p+1)}S_p(x_0, \Sigma).\label{eq:24}
\end{eqnarray}
Applying (\ref{eq:23}) to $v=\frac{\partial}{\partial z_i}$ for $i=p+1, \cdots, n$, and summing the obtained  inequalities  we have that \begin{equation}\label{eq:25}
\Ric_{1\bar{1}}+\Ric_{2\bar{2}}+\cdots +\Ric_{p\bar{p}}=S_p(x_0, \Sigma)+\sum_{\ell=p+1}^n \sum_{j=1}^p R_{\ell\bar{\ell}j\bar{j}}\le S_p(x_0, \Sigma).
\end{equation}
Combining (\ref{eq:24}) and (\ref{eq:25}) we have that
\begin{equation}\label{eq:26}
0<S^\perp_k(x_0)\le  S^{\perp}_p(x_0, \Sigma)\le S_p(x_0,\Sigma)-\frac{2}{p+1} S_p(x_0,\Sigma)=\frac{p-1}{p+1}S_p(x_0, \Sigma).
\end{equation}
This implies $S_p(x_0, \Sigma)>0$, a contradiction,  since we have shown that a consequence of (\ref{eq:23}) is  $S_p(x_0, \Sigma)\le 0$.

From the definition of $S_k^\perp$ it is easy to see that $\Ric^\perp>0$ implies that $S_k^\perp>0$ for all $k\in \{1, \cdots, n\}$. Hence $h^{p, 0}=0$ for all $p\ge 1$ by the above under the assumption $\Ric^\perp>0$.

 The simply-connectedness claimed in Theorem \ref{thm:1} follows from the argument of \cite{Ko} illustrated in the introduction. The proof under the assumption $S_k>0$ is similar, but easier in view of (\ref{eq:23}).
 
Note that under $\Ric^\perp>0$, $\pi_1(N)$ is finite by a result of \cite{NZ}. This in particular implies that $b_1=2h^{1, 0}=0$. The argument here provides an alternate proof of this.

\begin{remark} The argument here also provides an alternate  proof of the main theorem  of \cite{Ni-Zheng2}. It is clear that the K\"ahlerity is not absolutely needed. Hence one can easily formulate a corresponding result for Hermitian manifolds. We leave this to interested readers. The concepts of $S_k(x, \Sigma)$ and $S^\perp_k(x_0, \Sigma)$ were introduced  in \cite{NZ, N, Ni-ijm, Ni-Zheng2}.
\end{remark}

\section{Rational connectedness and $\Ric_k$}

A complex manifold $N$ is called rationally connected if any two points of $N$ can be joined by a chain of rational curves. Various criterion on the rational connectedness have been established by various authors.
In particular the following was prove in \cite{CDP}:
\begin{theorem}[Campana-Demailly-Peternell]\label{prop:61} Let $N$ be a projective algebraic manifold of complex dimension $n$. Then $N$ being rationally connected if and only if for any ample line bundle $L$, there exist $C(L)$ such that
\begin{equation}\label{eq-prop:61}
H^0(N, ((T'N)^*)^{\otimes p} \otimes L^{\otimes \ell})=\{0\}
\end{equation}
for any $p\ge C(L)\ell$, with $\ell$ being any positive integer.
\end{theorem}

It was proved in \cite{HeierWong} that a compact projective manifold with positive holomorphic sectional curvature must be rationally connected. The projectivity was proved in \cite{XYang} afterwards (an alternate proof of the rational connectedness was also given there). In \cite{N}, the concept $\Ric_k$ was introduced, which interpolates between the holomorphic sectional curvature and the Ricci curvature. Precisely for any $k$ dimensional subspace $\Sigma\subset T_x'N$, $\Ric_k(x, \Sigma)$ is the Ricci curvature of $R|_\Sigma$. Under $\Ric_k<0$, the $k$-hyperbolicity was proved in \cite{N}.

We say $\Ric_k(x)>\lambda(x)$ if $\Ric_k(x, \Sigma)(v, \bar{v})>\lambda |v|^2$, for any $v\in \Sigma$ and  for every $k$-dimensional subspace $\Sigma$. Similarly  $\Ric_k>0$ means that $\Ric_k(x)>0$ everywhere. The condition $\Ric_k>0$ does not become weaker as $k$ increases since more $v$ needs to be tested. In fact Hitchin \cite{Hitchin} illustrated examples of K\"ahler metrics  with $\Ric_1>0$  on all Hirzebruch surfaces.  But on most of them one could not possibly find metrics with  $\Ric_2>0$. More examples can be found in \cite{AHZ, NZ}. But it is easy to see that $S_k>0$ does follows from $\Ric_k>0$, and $S_k>0$ becomes weaker as $k$ increases with $S_1$ being the same as the holomorphic sectional curvature and $S_n$ being the scalar curvature.
Hence if $\Ric_2>0$, $N$ is also projective by the result of \cite{Ni-Zheng2}.  Naturally one would ask {\it whether or not   a compact K\"ahler manifold with $\Ric_k>0$ for some $k\in\{3, \cdots, n-1\}$ is  projective} since the projectivity has been known  for the cases of $k=1, k=2$ and $k=n$. The following result  provides an affirmative answer.

\begin{theorem}\label{thm:61-add}
Let $(N^n, h)$ be a compact K\"ahler manifold with $\Ric_k>0$ for some $1\le k \le n$. Then $h^{p, 0}=0$ for $1\le p\le n$. In particular, $N$ must be projective.
\end{theorem}
\begin{proof} By Theorem \ref{thm:1} and that $\Ric_k>0$ implies $S_k>0$  we have that $h^{p, 0}=0$ for $p\ge k$. Hence we only need to focus on the case $p<k$. The first part of proof of Theorem \ref{thm:1} asserts that if there exists a holomorphic $(p, 0)$-form $\phi\ne 0$,  then (\ref{eq:23}) holds. Namely  there exists $x_0\in N  $, and a unitary normal coordinate centered at  $x_0$ such that at $x_0$:
\begin{equation}\label{eq:623}
 \sum_{j=1}^p R_{v\bar{v}j\bar{j}}\le 0
\end{equation}
for any $v\in T_{x_0}'N$.

Now we pick a $k$-subspace $\Sigma\subset T_{x_0}'N$ such that it contains the $p$-dimensional subspace spanned by  $\{\frac{\partial}{\partial z^1}, \cdots, \frac{\partial}{\partial z^p}\}$. Then by the assumption $\Ric_k>0$,  $\forall j\in \{1, \cdots, p\}$.
$$
\aint_{v\in \mathbb{S}^{2k-1}\subset \Sigma} R_{v\bar{v}j\bar{j}}\, d\theta(v)=\frac{1}{k} \Ric_k\left(\frac{\partial}{\partial z^j}, \frac{\partial}{\partial z^{\bar{j}}}\right)>0.
$$
Thus we have that
$$
 \aint_{v\in \mathbb{S}^{2k-1}\subset \Sigma}\sum_{j=1}^p R_{v\bar{v}j\bar{j}}\, d\theta(v)>0.
$$
This is a contradiction to (\ref{eq:623}). The contradiction proves that $h^{p, 0}=0$ for $p<k$. The projectivity follows from $h^{2, 0}=0$ and a theorem of Kodaira (cf. \cite{MK}, Theorem 8.3 of Chapter 3).
\end{proof}
For $k=1, 2, n$, the result  are previously  known except when $k=2, p\ne 2$. The above proof provide a unified argument for all the previous known cases. The argument above proves a bit more. To state the result we introduce the following:

\begin{definition}
We call the curvature operator $R$ BC-$p$  positive at $x_0$ (BC stands for the bisectional curvature)  if for any unitary orthogonal $p$-vectors $\{E_1,\cdots, E_p\}$, there exists a $v\in T'_{x_0}N$ such that
\begin{equation}\label{eq:BC-pq}
\sum_{i=1}^p  R_{v\bar{v} E_i \overline{E}_i}  >0.
\end{equation}
We say that $(N, h)$ is BC-$p$  positive if the above holds all $x_0\in N$. This can be easily adapted to Hermitian bundle $(V, h)$ over a Hermitian manifold $N$ since the expression in (\ref{eq:BC-pq}) makes sense for $v\in V_{x_0}$ and  unitary $p$-vectors $\{E_1,\cdots, E_p\}$ of $T'_{x_0}N$.
\end{definition}

  It is easy to see that BC-$1$ positivity is the same as $RC$-positivity for the tangent bundle defined in \cite{XYang}. In general BC-$p$ positivity amounts to at any $x\in N$
\begin{equation}\label{eq:defbc}
\min_{\Sigma\in G_{p, n}(T_x'N)} \max_{|X|=1} \left(  \aint_{Z\in \mathbb{S}^{2p-1}\subset \Sigma} R(X,\overline{X}, Z, \overline{Z})\, d\mu(Z)\right)>0.
\end{equation}
Here $G_{p, n}(T_x'N)$ denotes the Grassmanian of rank $p$ subspaces of $T_x'N$.
 If we endow a compact complex manifold $N^n$ with a Hermitian metric. Let $R$ be its curvature, which can be viewed as section of $\bigwedge^{1, 1}(\mbox{End}(T_{x_0} N))$. Then BC-$p$ positivity can be defined for any Hermitian vector bundles. Here we  specialize it to $V=T'N$.

\begin{corollary}\label{coro:6-add-3} If the curvature of a Hermitian manifold $(N^n, h)$ satisfies the  BC-$p$ positivity for some $1\le p\le n$, then $h^{p, 0}=0$. Hence any K\"ahler manifold with BC-$2$ positive curvature must be projective. Moreover, the $2$-positivity of $\Ric_k$ (for some $k\ge 2$) implies the $BC$-$2$ positivity, thus the projectivity of $N$. Similarly $\Ric_{k+1}^{\perp}>0$ implies BC-$p$ positivity, $\forall p\ge k$.
\end{corollary}
In Theorem \ref{thm:6-added-ricciperp} the last statement is strengthen into {\it BC-$p$ positivity, $\forall p\ge 1$}.
Here we define $\Ric_k^{\perp}(x, \Sigma)(v, \bar{v})\doteqdot \Ric_k(x, \Sigma)(v, \bar{v})-H(v)/|v|^2$ for a $k$-dimensional subspace $\Sigma\subset T'_xN$ and $v\in \Sigma$. The positivity of $\Ric_k^{\perp}$ is defined as $\Ric_k$. Clearly $\Ric_1^{\perp}\equiv 0$. Note that $\Ric_2^{\perp}\ge 0$ is the same as the orthogonal bisectional curvature is nonnegative. $\Ric_n^{\perp}$ is the same as $\Ric^\perp$.

\begin{proof}
By the proof of Theorem \ref{thm:61-add} and the definition, we only need to show the last statement. The $2$-positivity of $\Ric_k$ means that for any $k$-dimensional $\Sigma\subset T_{x_0}'N$ and any two unitary orthogonal $E_1, E_2\in \Sigma$
$$
\Ric_k(x_0, \Sigma)(E_1, \overline{E}_1)+\Ric_k(x_0, \Sigma)(E_2, \overline{E}_2)>0.
$$
This clearly implies BC-$2$ positivity, since for any given unitary orthogonal $\{E_1, E_2\}$ there always a $k$-dimensional $\Sigma$ containing them, and
if $R(v,\bar{v}, E_1, \overline{E}_1)+R(v,\bar{v}, E_2, \overline{E}_2)\le 0, \forall v\in \Sigma$ it is easy to see that $\Ric_k(x_0, \Sigma)(E_1, \overline{E}_1)+\Ric_k(x_0, \Sigma)(E_2, \overline{E}_2)\le 0$.
If $\Ric^{\perp}_{k+1}>0$, it implies that for a unitary frame $\{E_i\}$ of a $k+1$-dimensional $\Sigma$ with $E_1=X/|X|$,
$$
\sum_{j=2}^{k+1} R(X, \overline{X}, E_j, \overline{E}_j)>0,
$$
which implies BC-$k$ positive. On the other hand, a simple calculation shows that
$$
\aint_{\mathbb{S}^{2k+1}\subset \Sigma\subset T_x'N} \Ric_{k+1}^\perp (Z, \overline{Z})\, d\theta(Z)=\frac{k}{(k+1)(k+2)} S_{k+1}(x).
$$
The case of $\ell>k$ follows by the proposition below if $k+1<n$.
\end{proof}

\begin{proposition}\label{prop:6-add-1} For a K\"ahler manifold $(N, h)$, $S_k(x_0)>0$ implies BC-$p$ positivity for any $p\ge k$, and  the $\ell$-positivity of $\Ric_k(x_0)$ (with $\ell\le k$) implies BC-$p$ positivivity for any $\ell\le p\le n$.
\end{proposition}
\begin{proof} Note that $S_p>0$ implies  BC-$p$ positivity. The first claim follows from  that $S_k>0  \implies S_p>0$ for any $p\ge k$. For the second statement, if  $p\ge k$, the result follows from the first. If $\ell\le p<k$, for  unitary $p$-vectors $\{E_1, \cdots, E_p\}$ we choose a $k$ subspace  $\Sigma$ containing them. If
$\sum_{i=1}^p R(v,\bar{v}, E_i, \overline{E}_i)\le 0, \forall v\in \Sigma$, it implies that
$\sum_{i=1}^p \Ric_k(x_0, \Sigma)(E_i, \overline{E}_i)\le 0$. This violates the  $\ell$-positivity of $\Ric_k(x_0)$.
\end{proof}

One can also extend the definition of $\Ric_k$ to a Hermitian vector bundle over Hermitian manifolds.  Let $R=R_{\alpha\bar{\beta}i}^{\quad j}dz^{\alpha}\wedge d\bar{z}^{\beta}\otimes e_i^*\otimes e_j$ be the curvature of a Hermitian vector bundle $(V, h)$ over a Hermitian manifold.

\begin{definition}\label{def:rick} Let $\Sigma\subset T'_{x_0}N$ and $\sigma\subset V_{x_0}$ be two $k$-dimensional subspaces.
Define for $X\in T'_{x_0}N$, $v=v^i e_i\in V_{x_0}$ with $\{e_k\}_{k=1}^{L}$ being a unitary frame of $V_{x_0}$, $L=\dim(V_{x_0})$, the first and second $\Ric_k$ as follows:
$$
\Ric^1_k(x_0, \sigma)(X, \overline{X})=\sum_{i=1}^k R_{X \overline{X}\, s}^{\quad\quad r} a^s_i\overline{a_i^t} h_{r\bar{t}}; \quad \Ric^2_k(x_0, \Sigma)(v, \bar{v})=\sum_{\alpha=1}^kR_{E_\alpha \overline{E}_{\alpha}\, i}^{\quad\quad j} v^i \bar{v}^l h_{j\bar{l}},
$$
with $\{E_\alpha\}_{\alpha=1}^k$ being a unitary frame of $\Sigma$, and  $\{\tilde{e}_i\}_{i=1}^k$ being a unitary frame of $\sigma$. Here $\tilde{e}_i=\sum_{k=1}^L a_i^k e_k$.
\end{definition}
Note that $\Ric^1$ is a $(1,1)$-form of $N$, and it coincides with the first Chern-Ricci of a Hermitian manifold if $k=n$ and  $V=T'N$. Observe that for $V=T'N$, $\left.\Ric^1_k\right|_{\sigma}$ is $\Ric_k(x_0, \sigma)$ when $N$ is K\"ahler, and generalizes $\Ric_k$ to the case of $N$ being just Hermitian.

The Corollary \ref{coro:6-add-3} generalizes the main theorem of \cite{Ni-Zheng2}. Towards the rational connectedness we prove the following result.

\begin{theorem}\label{thm:61} Let $(N^n, h)$ be a compact projective manifold with $\Ric_k>0$ for some $k\in \{1, \cdots, n\}$. Then (\ref{eq-prop:61}) holds, and  $N$ must be rationally connected.
\end{theorem}
\begin{proof} Before the general case, we start with a proof for the special case $k=1$ by proving the above criterion in Theorem \ref{prop:61} directly via the $\partial\bar{\partial}$-Bochner formula.
Let $s$ be a holomorphic section in $H^0(N, ((T'N)^*)^{\otimes p} \otimes L^{\otimes \ell})$. Locally it can be expressed as
$$
s=\sum_{I_p} a_{I_p, \ell}dz^{i_1}\otimes \cdots \otimes dz^{i_p} \otimes e^\ell
$$
with $I_p=(i_1, \cdots, i_p)\in \mathbb{N}^p$, and $e$ being a local holomorphic section of $L$ and $e^\ell =e\otimes \cdots\otimes e$ being the $\ell$ power of $e$. Equip $L$ with a Hermitian metric $a$ and let $C_a$ be the corresponding curvature form. The point-wise norm $|s|^2$ is with respect to the induced metric of $((T'N)^*)^{\otimes p}$ and $L^{\otimes \ell})$. The $\partial\bar{\partial}$-Bochner formula implies that for any $v\in T'_xN$:
\begin{eqnarray}
\partial_v\bar{\partial}_{\bar{v}} |s|^2 &=&|\nabla_v s|^2+\sum_{I_p}\sum_{t=1}^n\sum_{\alpha=1}^p \langle a_{I_p, \ell} R_{v\bar{v} i_\alpha \bar{t}}dz^{i_1}\otimes \cdots \otimes dz^{i_{\alpha-1}}\otimes dz^t\otimes\cdots\otimes dz^{i_p}\otimes e^{\ell}, \bar{s}\rangle \nonumber\\
&\quad&-\sum_{I_p} \langle a_{I_p, \ell} \ell C_a(v,\bar{v}) dz^{i_1}\otimes \cdots \otimes dz^{i_p} \otimes e^\ell, \bar{s}\rangle.\label{eq:61}
\end{eqnarray}
Applying the above equation at the point $x_0$,  where $|s|^2$ attains its maximum, with respect to a normal coordinate centered at $x_0$. Pick a unit vector $v$ such that $H(v)$ attains its minimum on $\mathbb{S}^{2n-1}\subset T_{x_0}'N$. By the assumption $H>0$, there exists a $\delta>0$ such that $H(v)\ge \delta$ for any unit vector and any $x\in N$.  Diagonalize $R_{v\bar{v}(\cdot)\overline{(\cdot)}}$ by a suitable chosen unitary frame $\{\frac{\partial}{\partial z^1}, \cdots, \frac{\partial}{\partial z^n}\}$.  Applying the first and  second derivative tests,   it shows that if at $v\in \mathbb{S}^{2n-1}$, $H(v)$ attains its minimum, then
$
R_{v\bar{v}w\bar{w}}\ge \frac{\delta}{2},
$
and $R_{v\bar{v}v\bar{w}}=0$,
for any $w$ with $|w|=1$, and $\langle w, \bar{v}\rangle =0$. This implies that
$$
R_{v\bar{v} i_\alpha \bar{i}_\alpha}=|\mu_1|^2 R_{v\bar{v}v\bar{v}}+|\beta_1|^2 R_{v\bar{v}w\bar{w}}\ge \frac{\delta}{2}
$$
 where we write $\frac{\partial}{\partial z^{i_\alpha}}=\mu_1 v+\beta_1 w$ with $|\mu_1|^2+|\beta_1|^2=1$, $w\in \{v\}^{\perp}$ and $|w|=1$. (This perhaps goes back to the work of Berger. See also for example \cite{XYang} or Corollary 2.1 of \cite{NZ}.)
If $A$ is the upper bound of $C_a(v, \bar{v})$, we have that
$$
0\ge \partial_v\bar{\partial}_{\bar{v}} |s|^2\ge \left(\frac{p\delta}{2}-\ell A\right) |s|^2.
$$
This is a contradiction for $p\ge \frac{3A \ell}{\delta}$ if $s\ne 0$. Hence we can conclude that for any $p\ge C(L)\ell$ with $ C(L)=\frac{3A}{\delta}$, $H^0(N, ((T'N)^*)^{\otimes p} \otimes L^{\otimes \ell})=\{0\}$.

For the general case, namely $\Ric_k>0$ for some $k\in \{1, \cdots, n\}$,  we combine the argument above with the second variation result of \cite{Ni-Zheng2}. At the point $x_0$ where the maximum of $|s|^2$ is attained, we pick $\Sigma$ such that $S_k(x_0, \Sigma)$ attains its minimum $\delta_1>0$. For simplicity of the notations, we denote the average of a function $f(X)$ over the unit sphere $\mathbb{S}^{2k-1}$ in $\Sigma$  by $\aint f(X)$. The second variation consideration in \cite{Ni-Zheng2} gives the following useful estimates.

\begin{proposition}[Proposition 3.1 of \cite{Ni-Zheng2}]\label{prop:31} Let $\{ E_1, \ldots , E_m\}$ be a unitary frame at $x_0$ such that  $\{E_i\}_{1\le i\le k}$ spans $\Sigma$.  Then for any $E\in \Sigma$, $E'\perp \Sigma$, and any $k+1\leq p\leq m$, we have
\begin{eqnarray}
\aint R(E,\overline{E}', Z, \overline{Z})d\theta(Z)=\aint R(E', \overline{E}, Z, \bar{Z})d\theta(Z)&=&0, \label{eq:62-1st} \\
 \aint R(E_p,\overline{E}_p, Z,   \overline{Z})\, d\theta(Z)&\ge&\frac{S_k(x_0,\Sigma)}{k(k+1)}. \label{eq:64}
\end{eqnarray}
\end{proposition}
\begin{proof} For the convenience of the reader we include the proof.  The proof  uses the first and second variation out of the fact that $S_k(x_0, \Sigma)$ is minimum.  Let $a\in \mathfrak{u}(m)$ be an element of the Lie algebra of $\mathsf{U}(m)$. Consider the function:
$$
f(t)=\aint H(e^{t a} X)\, d\theta(X).
$$
By the choice of $\Sigma$,  $f(t)$ attains its minimum at $t=0$. This implies that $f'(0)=0$ and $f''(0)\ge 0$. Hence
\begin{eqnarray}
&\, &\aint \left( R(a(X), \overline{X}, X, \overline{X}) +R(X, \bar{a}(\overline{X}), X, \overline{X})\right)\,
d\theta(X)=0; \label{eq:nz25}\\
&\,& \aint \left(R(a^2(X), \overline{X}, X, \overline{X}) +R(X, \bar{a}^2(\overline{X}), X, \overline{X})+4R(a(X),
\bar{a}(\overline{X}), X, \overline{X})\right) d\theta(X) \nonumber\\
&\, & + \aint \left( R(a(X), \overline{X}, a(X), \overline{X})+ R(X, \bar{a}(\overline{X}), X,
\bar{a}(\overline{X})\right) d\theta(X) \ge 0. \label{eq:nz26}
\end{eqnarray}
We exploit these by looking into some special cases of $a$. Let  $W\perp \Sigma$ and $Z\in \Sigma$ be two fixed vectors.
Let $a=\i \left(Z\otimes \overline{W}+W\otimes \overline{Z}\right)$. Then
$$
a(X)=\i \langle X, \overline{Z}\rangle W;\quad  a^2(X)=-\langle X, \overline{Z}\rangle Z.
$$
Applying (\ref{eq:nz26}) to the above $a$ and also the one with $W$ being replaced by $\i W$, and add the resulting two
estimates together,   we have that
\begin{equation}\label{eq:nz27}
4\aint |\langle X, \overline{Z}\rangle|^2R(W, \overline{W}, X, \overline{X}) d\theta(X)\ge \aint \langle X,
\overline Z\rangle  R(Z, \overline{X}, X, \overline{X}) +\langle Z, \overline{X}\rangle R(X, \overline{Z}, X,
\overline{X}).
\end{equation}

Applying the above to $Z\in \mathbb{S}^{2k-1}\subset \Sigma$ and taking  the average of the result we then have
$$
\frac{4}{k} \aint R(W, \overline{W}, X, \overline{X}) d\theta(X)\ge \frac{2}{k} \aint R(X, \overline{X}, X,
\overline{X}).
$$
This proves (\ref{eq:64}).
 By combining (\ref{eq:nz25}) (with $a$ as above) and the one with $W$ being replaced by $\i W$, we obtain two equalities:
$$\aint \langle X, \overline{Z}\rangle R(W, \overline{X}, X, \overline{X})=\aint \langle Z, \overline{X}\rangle R(X, \overline{W}, X, \overline{X})=0.
$$
Now write $X=x_1E_1+x_2 E_2+\cdots+x_k E_k$. Let $Z=E_i, W=E_{\ell}$ (for $i=1, 2$, $\ell\ge k+1$). Direct calculation (with $Z=E_1$) shows that
$$
\aint R_{\ell\bar{1}1\bar{1}} |x_1|^4+2\sum_{j=2}^k R_{\ell\bar{1}j\bar{j}}|x_1|^2|x_j|^2 =0.
$$
Applying the integral identities in the proof of the Berger's lemma (cf. Lemma 1.1 of  \cite{Ni-Zheng2}),
the above equation  (together with the case $Z=E_i$ with $2\le i\le k$) implies that
\begin{equation}\label{eq:nz27sub1}
\sum_{j=1}^k R_{\ell \bar{i}j\bar{j}}=0, \forall 1\le i\le k, k+1\le \ell \le n.
\end{equation}
This and its conjugate  imply (\ref{eq:62-1st}).
\end{proof}
As \cite{Ni-Zheng2}, we may choose the frame so that $\aint R_{v\bar{v} (\cdot)\overline{(\cdot)}}$ is diagonal.
Integrating (\ref{eq:61}) over the unit sphere $\mathbb{S}^{2k-1}\subset \Sigma$ we have that
\begin{eqnarray*}
0&\ge & \aint \partial_v\bar{\partial}_{\bar{v}} |s|^2\, d\theta(v)\ge \sum_{I_p} |a_{I_p, \ell}|^2 \aint \left(\sum_{\alpha=1}^p R_{v\bar{v} i_\alpha \bar{i}_\alpha}-\ell C_a(v, \bar{v})\right)\, d\theta(v).
\end{eqnarray*}
Here we have chosen a unitary frame $\{\frac{\partial}{\partial z^1}, \cdots, \frac{\partial}{\partial z^n}\}$ so that $\aint R_{v\bar{v}(\cdot)\overline{(\cdot)}} \, d\theta(v)$ is diagonal.

As in \cite{Ni-Zheng2}, decompose $\frac{\partial}{\partial z^i}$ into the sum of $\mu_i E_i\in \Sigma$ and $\beta_i E_i'\in \Sigma^\perp$ with $|E_i|=|E_i'|=1$ and $|\mu_i|^2+|\beta_i|^2=1$. If we denote the lower bound of $\Ric_k$ by $\delta_2>0$, by (\ref{eq:62-1st}) and (\ref{eq:64})
\begin{eqnarray*}
\aint R_{v\bar{v}1\bar{1}}\, d\theta(v) &=&|\mu_1|^2 \aint R_{v\bar{v}E_1\overline{E}_1}\, d \theta(v) +|\beta_1|^2\aint R_{v\bar{v}E_1'\overline{E}_1'}\, d\theta(v)\\
&= &\frac{|\mu_1|^2}{k} \Ric_k(E_1, \overline{E}_1)+ |\beta_1|^2\aint R_{v\bar{v}E_1'\overline{E}_1'}\, d\theta(v)\ge \frac{|\mu_1|^2}{k}\delta_2+\frac{|\beta_1|^2}{k(k+1)}\delta_1\\
&\ge&   \frac{\min{(\delta_1, \delta_2)}}{k(k+1)}.
\end{eqnarray*}
The above estimate holds for any $\frac{\partial}{\partial z^{i_\alpha}}$ as well. Hence combining  two estimates above  we have that
$$0\ge  \aint \partial_v\bar{\partial}_{\bar{v}} |s|^2\, d\theta(v)\ge \left(p \frac{\min{(\delta_1, \delta_2)}}{k(k+1)}-\ell A\right)|s|^2.
$$
The same argument as the special case $k=1$ leads to a  contradiction,  if $p\ge C(L) \ell$ for suitable chosen $C(L)$, provided that $s\ne 0$. This proves the vanishing theorem claimed in Theorem \ref{prop:61} for manifolds with $\Ric_k>0$. \end{proof}

The simply-connectedness part of Theorem  \ref{thm:4} follows from Theorem \ref{thm:61-add}, Theorem \ref{thm-diameter}, and the argument of \cite{Ko} (recalled in the introduction) via Hirzebruch's Riemann-Roch theorem. It can also be inferred from the rational connectedness and Corollary 4.29 of \cite{Deb}. It is expected that the construction via the projectivization in \cite{NWZ, YZ} would give more examples of K\"ahler manifolds with $\Ric_k>0$.

Regarding rational connectedness we should point out that there exists a recent work \cite{Yang18}, in which it was proved that if $T'N$ is {\it uniformly $RC$-positive} in the sense that for any $x\in N$, there exists a $X$ such that $R(X, \overline{X}, V, \overline{V})>0$ for any $V\in T_xN$, then $N$ is projective and rationally connected. As pointed out above, BC-2 positivity (which follows from the uniform RC-positivity) already implies the projectivity. The uniform RC-positivity is equivalent to $$\delta\doteqdot \min_{x\in N} \left(\max_{|X|=1, X\in T'_xN} \left(\min_{|V|=1, V\in T'_xN} R(X, \overline{X}, V, \overline{V})\right)\right)>0.$$  Hence one can derive Theorem \ref{prop:61} from (\ref{eq:61}) directly by letting $v=X$  with $X$ being the vector which attains the maximum in the above definition, and $p\ge \frac{2A}{\delta}$. This provides a direct proof of Theorem 1.3 in \cite{Yang18}.

 Since the boundedness of smooth Fano varieties (namely there are finitely many deformation types) was also proved in \cite{KMM}, it is natural to ask {\it  whether or not  the family of K\"ahler manifolds with $\Ric_k>0$ (for some $k$, particularly for $n$ large and $n-k\ne 0$ small) is bounded}. Before one proves that every K\"ahler manifold with $\Ric^\perp>0$ is Fano, it remains an interesting  future project to {\it investigate the rational connectedness of compact K\"ahler manifols with $\Ric^\perp>0$}. For manifold with  QB$>0$, as a simple consequence  of the results in  the next section   and the result of \cite{Cam, KMM} we have the following corollary.

 \begin{corollary}\label{coro:61}
Any compact K\"ahler manifold $(N, h)$ with quasi-positive QB, or more generally with quasi-positive $\Ric^\perp$, and $\rho(N)=1$ must be rationally connected.
 \end{corollary}
The same conclusion  holds if $\Ric^\perp\ge 0$, $(N^n, h)$ is locally irreducible  and $\rho(N)=1$.

\section{Compact K\"ahler manifolds with $h^{1, 1}=1$ and CQB}

Recall the following result from \cite{NZ}, which is a consequence of a formula of Berger.
\begin{lemma} \label{lm:Gold} Let $(N^n, h)$ be a K\"ahler manifold of complex dimension $n$.  At any point $p\in N$, \begin{equation}\label{eq:scalar-ricciperp}
\frac{n-1}{n(n+1)}S(p)=\frac{1}{Vol(\mathbb{S}^{2n-1})}\int_{|Z|=1, Z\in T_p'N} \Ric^{\perp}(Z, \overline{Z})\, d\theta(Z)
\end{equation}
where $S(p)=\sum_{i=1}^n \Ric(E_i, \overline{E}_i)$ (with respect to any unitary frame $\{E_i\}$) denotes the scalar curvature at $p$
\end{lemma}

Note that the first Chern form $c_1(N)=\frac{\sqrt{-1}}{2\pi} r_{i\bar{j}} dz^i\wedge d\overline{z}^j$, with $r_{i\bar{j}}=\Ric(\frac{\partial}{\partial z^i}, \frac{\partial}{\partial \overline{z}^j})$.  Let $\omega_h=\frac{\sqrt{-1}}{2\pi} h_{i\bar{j}}$ be the K\"ahler form (the normalization is to make the K\"ahler and Riemannian settings coincide).  A direct computation via a unitary frame gives
\begin{equation}\label{eq:42}
c_1(N)(y)\wedge \omega_h^{n-1}(y)=\frac{1}{n}S(y)\, \omega_h^n(y).
\end{equation}
We also let $V(N)=\int_N \omega_h^n$. The normalization above makes sure that the volume of an algebraic subvariety has its volume being an integer.

Recall that for any line bundle $L$ its degree $d(L)$ is defined as
\begin{equation}\label{eq:43}
d(L)=\int_N c_1(L)\wedge \omega_h^{n-1}.
\end{equation}
When $h^{1,1}(N)=1$, it implies that $[c_1(N)]=\ell [\omega_h]$ for some constant  $\ell$. Hence we have that
$d(K_N^{-1})= \ell V(N)$.

Under the assumption (i) of Theorem \ref{thm:c1}, we know that $S(y)>0$ somewhere and $S(y)\ge 0, \forall x\in N$ by Lemma \ref{lm:Gold}, which then implies that  $d(K^{-1}_N)>0$, hence $\ell >0$. This shows that $[c_1(N)]>0$. Now Yau's solution to the Calabi's conjecture \cite{Yau,Tian}  implies that $N$ admits a K\"ahler metric such that its Ricci curvature is  $\ell \omega_h>0$.

The proof for statement (ii) is similar. The existence of negative K\"ahler-Einstein metric follows from the Aubin-Yau theorem \cite{Yau, Tian}.

To prove Corollary \ref{coro:11} we observe that if $\ell=0$ in the above argument, it implies that
$S(y)\equiv 0$. Hence by Lemma \ref{lm:Gold} we have that $\Ric^\perp\equiv 0$. By Theorem 6.1 of \cite{NWZ} it implies that $N$ is flat for $n\ge 3$, or $n=2$ and $N$ is either flat or locally a product. This contradicts to the assumption of local irreducibility.

Note that the same argument can be applied to conclude the same result $\Ric_k$ and $\Ric_k^{\perp}$.

\begin{proposition}\label{prop:simple}Let $(N, h)$ be a compact K\"ahler manifold of complex dimension $n$.
 Assume further that $h^{1, 1}(N)=1$. Then  (i) if $\Ric_k$ (or $\Ric^\perp_k$) is quasi-positive for some $1\le k\le n$ , then $N$ must be Fano;  (ii) if $\Ric_k$  (or $\Ric^\perp_k$) is quasi-negative, $N$ must be projective with ample canonical line bundle $K_N$. In particular in the case (i) $N$ admits a K\"ahler metric with positive Ricci, and in the case of (ii) $N$ admits a K\"ahler-Einstein metric with negative Einstein constant.
\end{proposition}

 Before introducing two new curvatures we first observe that in (\ref{eq:12}) if we replace $A$ by its traceless part $\AA=A-\lambda \id$ with $\lambda=\frac{\operatorname{trace}(A)}{n}$, it remains the same. Namely  QB$(A)$=QB$(\AA)$. Hence QB is defined on the quotient space $S^2(\mathbb{C}^n)/\{\mathbb{C} \id\}$, with $S^2(\mathbb{C}^n)$ being the space of Hermitian symmetric transformations of $\mathbb{C}^n$.  Now QB$>0$ means that QB$(A)>0$ for all $A\ne 0$ as an equivalence class.  This suggests a refined positivity QB$_k>0$, for any $1\le k\le n$,  defined as QB$(A)>0$ for any  $A\notin \{\mathbb{C}\id\}$ of rank not greater than $k$. Clearly for $k<n$, a nonzero Hermitian symmetric matrix with rank no greater than $k$ can not be in  $\{\mathbb{C} \id\}$. It is easy to see QB$_1>0$ is equivalent to $\Ric^\perp>0$ and QB$_n>0$  is equivalent to QB$>0$. Naturally a possible approach towards the classification of $\Ric^\perp>0$ is through the family of K\"ahler manifolds with QB$>0$ and  QB$_k>0$.

Now we  discuss the first of two new curvatures.  Recall that  {\it cross quadratic bisectional curvature}  CQB,  is defined as a Hermitian quadratic form on linear maps $A: T''N \to T'N$:
\begin{equation}\label{eq:34}
CQB_R(A)=\sum_{\alpha, \beta=1}^nR(A(\overline{E}_\alpha), \overline{A(\overline{E}_\alpha)}, E_\beta, \overline{E}_\beta)-R(E_\alpha, \overline{E}_\beta, A(\overline{E}_\alpha), \overline{A(\overline{E}_\beta)})
\end{equation}
for any unitary frame $\{E_\alpha\}$ of $T'N$. This is similar to (\ref{eq:12}). But here we allow $A$ to be any linear maps.  We say $R$ has CQB$>0$ if CQB$(A)>0$ for any $A\ne 0$.
For any $X\ne 0$, if we choose $\{E_\alpha\}$ with $E_1=\frac{X}{|X|}$, and let $A$ be the linear map satisfying $A(\overline{E}_1)=E_1$ and $A(\overline{E}_\alpha)=0$ for any $\alpha \ge 2$, it is easy to see that
CQB$_R(A)=\Ric^\perp(X, \overline{X})/|X|^2$. Hence CQB$>0$ implies that $\Ric^\perp>0$\footnote{In \cite{NZ-cross}, it was shown  that CQB$>0$ implies $\Ric>0$.}.  Theorem \ref{thm:NWZ1} below shows that CQB$>0$ holds for all classical K\"ahler C-spaces with $b_2=1$, unlike QB, which fails to be  positive on about $20\%$ of K\"ahler $C$-spaces with $b_2=1$. The expression CQB  is motivated by the work of Calabi-Vesentini \cite{CV} where the authors studied the deformation rigidity of compact quotients of Hermitian symmetric spaces of noncompact type. We can introduce the concept CQB$_k>0$  (or CQB$_k<0$), defined as CQB$(A)>0$ for any $A$ with rank not greater than $k$. Express $A$ as $\sum_{s=1}^k X_s\otimes Y_s$, then CQB$_k>0$ is equivalent to
\begin{equation}\label{eq:cqb_k}
\sum_{s, t=1}^k \Ric( X_s, \overline{X}_t) \langle Y_s, \overline{Y}_t\rangle -R(X_s, \overline{X}_t, Y_s, \overline{Y}_t)>0,\quad  \forall \sum_{s=1}^k X_s\otimes Y_s \ne 0.
\end{equation}

\begin{proposition} (i)The condition CQB$_1>0$ implies $\Ric^\perp>0$,  in particular $N$ satisfies $h^{p, 0}=0$, $\pi_1(N)=\{0\}$,  and $N$ is projective. (ii) If $N$ is compact with $n\ge 2$, and CQB$_2>0$, then Ricci curvature is $2$-positive.
\end{proposition}
\begin{proof} Part (i) is proved in the paragraph above together with  Theorem \ref{thm:1}. For part (ii), for any unitary frame $\{E_\alpha\}$,  let $A$ be the map defined as $A(\overline{E}_1)=E_2$ and $A(\overline{E}_2)=-E_1$, and $A(E_\alpha)=0$ for all $\alpha>2$. Then the direct checking  shows that CQB$>0$ is equivalent to
$$
\Ric(E_1, \overline{E}_1)+\Ric(E_2, \overline{E}_2)>0.
$$
Since this holds for any unitary frame we have the $2$-positivity of the Ricci curvature.
\end{proof}

The work of Calabi-Vesentini \cite{CV} proves   the following result.

\begin{theorem}\label{thm:CV1} Let $(N, h)$ be a compact K\"ahler manifold with quasi-negative CQB (namely CQB$\le0$ and $<0$ at least at one point). Then $$H^1(N, T'N)=\{0\}.$$ In particular, $N$ is deformation rigid in the sense that it does not admit nontrivial infinitesimal holomorphic deformation.
\end{theorem}
\begin{proof}
Let $\phi=\sum_{i, \alpha =1}^n\phi^i_{\bar{\alpha}}dz^{\bar{\alpha}}\otimes E_i$ be a $(0,1)$-form taking value in $T'N$ with $\{E_i\}$ being a local holomorphic basis of $T'N$. The Arizuki-Nakano  formula gives
\begin{equation}\label{eq:siu1}
(\Delta_{\bar{\partial}} \phi-\Delta_{\partial} \phi)^i_{\bar{\beta}}=R_{j\, \, \, \, \bar{\beta}}^{\,\, i  \bar{\tau}}\phi^j_{\bar{\tau}}-
\Ric^i_j\phi^j_{\bar{\tau}}.
\end{equation}
Under a normal coordinate we have that
$$
\langle \Delta_{\bar{\partial}} \phi-\Delta_{\partial} \phi, \overline{\phi}\rangle =-\left( \Ric_{j\bar{i}} \phi^j_{\bar{\beta}}\overline{\phi^i_{\bar{\beta}}}-R_{j\bar{i}\tau \bar{\beta}}\phi^j_{\bar{\tau}}\overline{\phi^i_{\bar{\beta}}}\right).
$$
Hence if $\Delta_{\bar{\partial}}\phi=0$, we then have
$$
0=\int_N |\partial \phi|^2+\int_N |\partial^* \phi|^2-\int_N \left( \Ric_{i \bar{j}} \phi^i_{\bar{\beta}}\overline{\phi^j_{\bar{\beta}}}-R_{j\bar{i}\tau \bar{\beta}}\phi^j_{\bar{\tau}}\overline{\phi^i_{\bar{\beta}}}\right).
$$
Letting $A(\overline{E}_\beta)=\phi^i_{\bar{\beta}} E_i$, the assumption amounts to that the expression in the third   integral above is non-positive and  negative over a open subset $U$ where CQB$<0$ if $(\phi^i_{\bar{\tau}})\ne 0$ on $U$.  This forces
$(\phi^i_{\bar{\beta}})\equiv 0$ on $U$, hence $\phi=0$  by the unique continuation since $\phi$ is harmonic.
\end{proof}
It has been proved in \cite{NZ} that if $\Ric^\perp<0$, then $H^0(N, T'N)=\{0\}$.
 By Table 1 of \cite{CV} and the proof of Theorem \ref{thm:NWZ1} below, all locally Hermitian symmetric spaces of noncompact type satisfy  CQB$<0$.  Moreover the above theorem generalizes the result of Calabi-Vesentini since there are examples of non Hermitian symmetric manifolds with CQB$<0$. Flipping the sign we  have the following corollary.

\begin{corollary}Let $(N, h)$ be a compact K\"ahler manifold with quasi-positive CQB. Then $$\mathcal{H}^{0, 1}_{\partial}(N, T'N)=\mathcal{H}^{1, 0}_{\bar{\partial}} (N, \Omega)= H^0(N, \Omega^1(\Omega))=\{0\},$$
where $\Omega =(T'N)^*$. If only $\Ric^\perp>0$ is assumed, then $H^0(N, \Omega)=\{0\}$.
\end{corollary}
In fact we can strengthen the argument to prove the following result.

\begin{theorem}\label{thm:cqb-rc} Assume that $(N, h)$ is a compact K\"ahler manifold with CQB$>0$. Then for any ample line bundle $L$, there exist $C(L)$ such that
\begin{equation}\label{eq-thm:cqb-rc}
H^0(N, ((T'N)^*)^{\otimes p} \otimes L^{\otimes \ell})=\{0\}
\end{equation}
for any $p\ge C(L)\ell$, with $\ell$ being any positive integer. In particular $N$ is rationally connected.
\end{theorem}
\begin{proof} First observe that a holomorphic section of $((T'N)^*)^{\otimes (p+1)} \otimes L^{\otimes \ell}$ can be viewed as a holomorphic $(1,0)$ form valued in  $((T'N)^*)^{\otimes p} \otimes L^{\otimes \ell}$. Write it as $\varphi=\varphi^{I_p}_{\alpha} dz^\alpha \otimes dz^{i_1}\otimes dz^{i_2}\otimes \cdots \otimes dz^{i_p} \otimes e^{\ell}$. Applying the Arizuki-Nakano formula to the $\bar{\partial}$-harmonic $\varphi$ as above, using the formula for the curvature of the tensor products, and under a normal coordinate, we have that
\begin{eqnarray*}
0\le \langle \square_{\partial}\varphi, \bar{\varphi}\rangle  \le \int_M \left(\Omega ^I_J \varphi^I_\alpha \overline{\varphi^J_\alpha}- \Omega^I_{J\, \gamma \bar{\alpha}}\varphi^I_\alpha \overline{\varphi^J_\gamma} \right)+A\ell |\varphi|^2\le  \int_M \left(-p\delta |\varphi|^2+A\ell |\varphi|^2\right)
\end{eqnarray*}
where $\Omega^I_{J\, \gamma \bar{\alpha}}dz^\gamma\wedge\bar{z}^\alpha$ is the curvature of $((T'N)^*)^{\otimes p}$ and $\Omega^I_J$ is the corresponding mean curvature, $\delta>0$ is the lower bound of CQB, $A$ is an upper bound of the scalar curvature of $L$ (equipped with a Hermitian metric of positive curvature). This implies that $\varphi=0$ if $p/\ell$ is sufficiently large, hence the result.
\end{proof}

Recently it was proved that CQB$>0$ implies that $M$ is Fano, which gives an alternate proof of the above result.

The results above  naturally lead to the following questions {\it (Q1): (a)
 Does  $H^1(N, T'N)=\{0\}$ hold under the weaker assumption that $\Ric^\perp<0$? (b) Is a harmonic map $f$ of sufficiently high rank from a K\"ahler manifold $(M, g)$ into a compact manifold with negative CQB  must be holomorphic or conjugate holomorphic? (c) Is there any nonsymmetric (locally) example of manifolds with CQB $<0$? (d) Do all K\"ahler C-spaces (the canonical K\"ahler metric) with $b_2=1$ satisfy  CQB$>0$ (below we provide a partial  answer to this)?  } The ultimate goal is to prove  a classification theorem for  compact K\"ahler manifolds with CQB$>0$.

  Concerning (c), in a recent preprint \cite{NZ-cross}  a nonsymmetric example has been constructed. Moreover examples show that $b_2$ can be arbitrarily large under CQB$>0$ condition. Concerning (d), we have the following affirmative answer for all classical K\"ahler C-spaces.

\begin{theorem}\label{thm:NWZ1}Let $N^n$ be a compact Hermitian symmetric space ($n\ge 2$), or classical K\"ahler C-space with $n\geq 2$ and $b_2=1$. Then the K\"ahler-Einstein metric  (unique up to constant multiple) has CQB$>0$.
\end{theorem}
\begin{proof} If we write $A(\overline{E}_\beta)=A^i_\beta E_i$, it is easy to see if we change to a different unitary frame $\overline{\widetilde{E}}_\alpha=B^\beta_\alpha \overline{E}_\beta$, the effect on $A$ is $BA B^{tr}$ with $B$ being a unitary transformation. Now
$$
CQB(A)=\Ric_{i \bar{j}}A^i_{\beta}\overline{A^j_{\beta}}-R_{j\bar{i}\tau \bar{\beta}}A^j_{\tau}\overline{A^i_{\beta}}.
$$
Given that there exists a normal form under congruence for symmetric  and  skew symmetric matrices, it is  helpful to write $A$ into sum of the symmetric and skew-symmetric parts.
For the special case $\Ric=\lambda h$, namely the metric is K\"ahler-Einstein with $\lambda>0$, if we decompose $A$ into the symmetric part $A_1$ and the skew-symmetric part $A_2$, noting that $R_{j\bar{i}\tau \bar{\beta}}$ is symmetric in $j, \tau$ and $i, \beta$ we have
$$
CQB(A)=\lambda|A_1|^2+\lambda|A_2|^2 -R_{j\bar{i}\tau \bar{\beta}}(A_1)^j_{\tau}\overline{(A_1)^i_{\beta}}\ge \lambda|A_1|^2-R_{j\bar{i}\tau \bar{\beta}}(A_1)^j_{\tau}\overline{(A_1)^i_{\beta}}.
$$
Now note that  $R_{j\bar{i}\tau \bar{\beta}}(A_1)^j_{\tau}\overline{(A_1)^i_{\beta}}$ is the Hermitian symmetric action $Q$ on the symmetric tensor (matrix) $A$ considered in \cite{Itoh} and \cite{CV}. Precisely  $Q$ is defined by
 $$
 Q(X\cdot Y, \overline{Z\cdot W})=R_{X\overline{Z} Y\overline{W}}
 $$
 for $X\cdot Y=\frac{1}{2}(X\otimes Y+Y\otimes X)$, and then is extended to all symmetric tensors.
Let $\nu$ denotes the biggest eigenvalue of $Q$. As in \cite{NWZ}, to verify the result we just need to compare $\lambda$ and $\nu$. This can be done for all Hermitian symmetric spaces by Table 2 in \cite{CV}. Note that $\lambda$ here is $\frac{R}{2n}$ in Calabi-Vesentini's paper \cite{CV}.  For the classical homogeneous examples which are not Hermitian symmetric we can use the comparison done in \cite{NWZ} with the data supplied by \cite{Itoh} and \cite{ChauTam}. If we use the notation of \cite{Itoh} and \cite{ChauTam}, only the three types below need to be checked:
$$
(B_r, \alpha_i)_{r\ge 3, 1<i<r}; \quad (C_r, \alpha_i)_{r\ge 3, 1<i<r}; \quad (D_r, \alpha_i)_{r\ge 4, 1< i< r-1}.
$$
The verification in Section 2 of \cite{NWZ} applies verbatim.
\end{proof}

The above result strengthens the one in \cite{NWZ}  since CQB$>0$ is stronger than  $\Ric^\perp>0$. Note that the result also holds for the exceptional (non-Hermitian symmetric) K\"ahler C-space $(F_4, \alpha_4)$ since for such a space $\lambda=11/2$ and the biggest eigenvalue of $Q$ is 1. A future  natural project  is to classify all the compact K\"ahler manifolds with CQB$>0$. The example in \cite{NWZ}, shows that there certainly are compact K\"ahler manifolds with $\Ric^\perp>0$, but not homogeneous.

The second related curvature is a dual version of CQB, which is  motivated by the study of the compact dual of the noncompact Hermitian symmetric spaces in \cite{CV}. We denote  it by $^d$CQB. It  is defined as a quadratic Hermitian form on the space of linear maps $A:T'N \to T''N$:
$$
^dCQB_R(A)\doteqdot R(\overline{A(E_i)}, A(E_i), E_k, \overline{E}_k)+R(E_i, \overline{E}_k, \overline{A(E_i)}, A(E_k)).
$$
Similarly we can  introduce the concept $^d$CQB$_k>0$. The analogy of $\Ric^\perp$ is
$$\Ric^{+}(X, \overline{X})\doteqdot \Ric(X, \overline{X})+H(X)/|X|^2.$$
 We say  $^d$CQB$_k>0$ if $^d$CQB$(A)>0$ for any $A\ne0$ with rank not greater than $k$. Letting $A$ be the map which satisfies $A(E_1)=\overline{E}_1$ and $A(E_i)=0$ for all $i\ge 2$,
it is easy to see that $^d$CQB$_1>0$ implies that $\Ric^{+}>0$. We discuss geometric implications of these two curvature notions in details next.

\section{Manifolds with $^d$CQB$>0$ and  the properties of $\Ric^+$}

We have seen that $\Ric^{\perp}_k$ interpolates between the orthogonal bisection sectional curvature and $\Ric^\perp$. In a similar manner we can define $\Ric_k^+$, which interpolates between the holomorphic sectional curvature and $\Ric^+$ as $\Ric_k$ does.
First we show that he diameter estimate  in \cite{NZ} for manifolds with $\Ric^\perp$  can be extended to $\Ric^+$ ($\Ric_k$, $\Ric^+_k$).
The argument via the second variational formulae in the proof of Bonnet-Meyer theorem proves the compactness of the K\"ahler manifolds if the $\Ric^{+}$ is uniformly bounded from below by a positive constant.

\begin{theorem}\label{thm-diameter}
Let $(N^n, h)$ be a K\"ahler manifold with $\Ric^{+}(X, \overline{X})\ge (n+3) \lambda |X|^2$ with $\lambda>0$. Then $N$ is compact with diameter bounded from the above by $\sqrt{\frac{2n}{(n+3)\lambda}}\cdot \pi$. Moreover, for any geodesic $\gamma(\eta): [0, \ell]\to N$ with length $\ell>\sqrt{\frac{2n}{(n+3)\lambda}}\cdot \pi$, the index $i(\gamma)\ge 1$.

Similarly, for  a K\"ahler manifold  $(N^n, h)$ with $\Ric_k\ge (k+1)\lambda>0$, its diameter is bounded from above by
$\sqrt{\frac{2k-1}{(k+1)\lambda}}\cdot \pi$; For a K\"ahler manifold  $(N^n, h)$ with $\Ric^\perp_k\ge (k-1)\lambda>0$ its diameter is bounded from above by
$\sqrt{\frac{2}{\lambda}}\cdot \pi$; For a K\"ahler manifold  $(N^n, h)$ with $\Ric^+_k\ge (k+3))\lambda>0$ its diameter is bounded from above by $\sqrt{\frac{2k}{(k+3)\lambda}}\cdot \pi$.
\end{theorem}

Note that the result (for $\Ric^+$) is slightly better than $\sqrt{\frac{2n-1}{(n+1)\lambda}}\pi$, the one predicted by the Bonnet-Meyer estimate assuming $\Ric(X, \overline{X})\ge (n+1)\lambda |X|^2$ for $n\ge2$.  But it is roughly about $\sqrt{2}$ times the one predicted by the Tsukamoto's theorem in terms of the lower bound of the holomorphic sectional curvature. Let $N=\mathbb{P}^1\times \cdots \times \mathbb{P}^1$, namely the product of  $n$  copies of $\mathbb{P}^1$, its diameter is  $\sqrt{\frac{n}{2}}\pi$. An easy computation shows that it has $\Ric= 2$ and $H\ge \frac{2}{n}$. This shows that the upper bound provided by Tsukamoto's theorem holds equality on both $\mathbb{P}^n$ and $N=\mathbb{P}^1\times \cdots \times \mathbb{P}^1$. The product of $n$-copies of $\mathbb{P}^1$  also illustrates a compact K\"ahler manifold (after proper scaling)  with $\Ric= n+1$, but its diameter is roughly about $\sqrt{2}$ times of that of $\mathbb{P}^n$. The product example and $\mathbb{P}^n$ indicate that the above estimate on the diameter is far from being sharp.

We prove Theorem \ref{thm:3} via a  vanishing theorem with weaker assumptions. For that we introduce the scalar curvatures $S^+(x,\Sigma)$ which is defined as
$$
S^{+}_k(x, \Sigma)=k \aint_{Z\in \Sigma, |Z|=1} \Ric^{+}(Z, \overline{Z})\, d\theta(Z)
$$
for any $k$-dimensional subspace $\Sigma\subset T_xN$. Similarly we say $S_k^+>0$ if $S_k^+(x, \Sigma)>0$ for any $x$ and $\Sigma$.

\begin{theorem} Assume that $S_k^+>0$, then $N$ is BC-$p$, positive and  $h^{p, 0}=0$ for $k\le p\le n$.
\end{theorem}
\begin{proof} The first part of proof follows similarly as in that of Theorem \ref{thm:1}. Assuming the existence of a nonzero holomorphic $(p, 0)$-form $\phi$ leads to the  conclusion that at the point $x_0$ where the maximum of the comass $\|\phi\|_0$ is attained we have that
\begin{equation}\label{eq:6-help-contrary}
0\ge \sum_{j=1}^p R_{v\bar{v} j\bar{j}}
\end{equation}
for any $v\in T'_{x_0} N$, for a particularly chosen  frame $\{\frac{\partial}{\partial z_\ell}\}_{\ell=1, \cdots, n}$ with $\Sigma=\operatorname{span}\{\frac{\partial}{\partial z_1}, \cdots,  \frac{\partial}{\partial z_p}\}$. This  implies that $S_p(x_0, \Sigma)\le 0$, by applying the above to $v=\{\frac{\partial}{\partial z_i}\}_{1\le i\le p}$.

Now a similar calculation as that of Section 2 shows that \begin{eqnarray}
\frac{1}{p} S^{+}_p(x_0, \Sigma)&=&\aint_{Z\in \Sigma, |Z|=1} \Ric^+(Z, \overline{Z})\, d\theta(Z)=\aint_{Z\in \Sigma, |Z|=1} \left(\Ric(Z, \overline{Z})+H(Z)\right)\, d\theta(Z)\nonumber\\
&=&\aint \frac{1}{Vol(\mathbb{S}^{2n-1})}\left(\int_{\mathbb{S}^{2n-1}} \left(nR(Z, \overline{Z}, W, \overline{W})+H(Z)\right)\, d\theta(W)\right)\, d\theta(Z)\nonumber\\
&=&\frac{1}{Vol(\mathbb{S}^{2n-1})}\int_{\mathbb{S}^{2n-1}}\left( \aint \left(nR(Z, \overline{Z}, W, \overline{W})+H(Z)\right)\, d\theta(Z)\right)d\theta(W)\nonumber\\
&=& \frac{1}{p}\left(\Ric_{1\bar{1}}+\Ric_{2\bar{2}}+\cdots +\Ric_{p\bar{p}} \right)+\frac{2}{p(p+1)}S_p(x_0, \Sigma).\label{eq:3-42}
\end{eqnarray}
Using the estimate (\ref{eq:23}) similarly as in Section 2 (cf.  (\ref{eq:25})) we have that
$$
\Ric_{1\bar{1}}+\Ric_{2\bar{2}}+\cdots +\Ric_{p\bar{p}} \le S_p(x_0, \Sigma ).
$$
Thus together with (\ref{eq:3-42}) it implies that
$$
 0< S^+_k(x_0)\le S^{+}_p(x_0, \Sigma)\le \frac{p+3}{p+1} S_p(x_0, \Sigma).
$$
This is a contradiction.
\end{proof}
Theorem \ref{thm:3} follows from the above theorem since $\Ric^+>0$ implies that $S^+_p>0$ for all $1\le p\le n$.
 In the theorem below we extend the projectivity part of Theorem \ref{thm:4} to $\Ric^{\perp}_k$ and $\Ric^+_k$.

\begin{theorem}\label{thm:6-added-ricciperp} Let $(N, h)$ be a compact K\"ahler manifold. If $\Ric^\perp_k>0$, or $\Ric^+_k>0$ for some $1\le k\le n$, then $N$ is BC-$p$ positive for $p\ge 2$. In particular $h^{p, 0}=0$ for $1\le p\le n$, and   $N$ is projective and simply-connected.
\end{theorem}
\begin{proof} We only provide the proof for $\Ric^\perp_k$ since the proof for $\Ric^{+}_k$ is similar. As before it suffices to prove that $N$ is BC-$p$ positive for all $1\le p\le n$ since $h^{p, 0}=0$ follows from this. By Theorem \ref{thm:1} and Corollary \ref{coro:6-add-3}, the BC-$p$ positivity (for all $1\le p\le n$) is known for $k=2, n$. Thus  we only need to prove it for $3\le k\le n-1$. By Corollary \ref{coro:6-add-3} again we have BC-$p$ positivity for $p\ge k-1$. Hence we  only need to prove it for $p\le k-2$. 

The proof is essentially the same as the proof of Theorem \ref{thm:1}. We prove by the contradiction argument.
Assume that there exists  unitary $p$-vectors $\{E_1, \cdots, E_p\}$ such that (\ref{eq:6-help-contrary}) hold for any $v\in T_x'N$.  Let $\Sigma=\operatorname{span}\{E_1, \cdots, E_p\}$. Since $k-2\ge p$ we extend them into unitary $k$-vectors $\{E_1, \cdots, E_k\}$. Let $\Sigma'$ be the $k$-dimensional subspace spanned by them. Clearly $\Sigma \subset \Sigma'$. We denote by $\Ric'$ the Ricci curvature restricted to $\Sigma'$. We also define similarly $(\Ric^\perp)'$, $S'_p$ and $(S^\perp)'_p$ correspondingly. In particular $(\Ric^\perp)'$ is $\Ric^{\perp}_k(x, \Sigma')$, $(S^\perp)'_p$ is the average of $(\Ric^\perp)'$ on the unit sphere of a $p$-dimensional subspace of $\Sigma'$, and  $S_p'(x,\Sigma)=\sum_{i, j=1}^p R_{i\bar{i}j\bar{j}}$.
Now the proof of Theorem \ref{thm:1} implies that for any $x\in N$
$$
(S^{\perp})'_p(x)\le \frac{p-1}{p+1} S'_p(x, \Sigma) = \frac{p-1}{p+1}\sum_{i, j=1}^p R_{i\bar{i}j\bar{j}}\le 0.
$$
On the other hand, the assumption implies that  $(\Ric^\perp)'>0$, hence $(S^{\perp})'_p>0$. This induces a contradiction. Hence we have $N$ is BC-$p$ positive for any $1\le p\le n$ if $\Ric^\perp_k>0$.
\end{proof}

Applying argument similar to that of the last section we also have the following result.

\begin{proposition}\label{prop:41-ct-Ric+} Let $(N^n, h)$ be a compact K\"ahler manifold of complex dimension $n$ with quasi-positive $\Ric^+$ (or $\Ric^+_k$). Assume further that $h^{1,1}(N)=1$ (or $\rho(N)=1$). Then $c_1(N)>0$, namely $N$ is Fano.
\end{proposition}
The proof of this result and the following lemma, which plays  the analogue role of  Lemma \ref{lm:Gold}, is the same as that of the last section.

\begin{lemma}Let $(N^n, h)$ be a K\"ahler manifold of complex dimension $n$.  At any point $p\in N$, \begin{equation}\label{eq:scalar-ricci+}
\frac{n+3}{n(n+1)}S(p)=\frac{1}{Vol(\mathbb{S}^{2n-1})}\int_{|Z|=1, Z\in T_p'N} \Ric^{+}(Z, \overline{Z})\, d\theta(Z)
\end{equation}
where $S(p)=\sum_{i=1}^n \Ric(E_i, \overline{E}_i)$ (with respect to any unitary frame $\{E_i\}$) denotes the scalar curvature at $p$.
\end{lemma}

Following the argument in the Appendix of \cite{NWZ} we also have that {\it  a $\Ric^+$-Einstein K\"ahler metric must be of constant curvature}. In particular, the one with zero scalar curvature must be flat. Hence we have the same result as Corollary \ref{coro:11} if we replace $\Ric^\perp$ by $\Ric^+$.

\begin{corollary}\label{coro:43}
Let $(N, h)$ be a compact K\"ahler manifold of complex dimension $n$ with $\Ric^{+}\ge 0$ (or $\Ric^{+}_k\ge 0$).
 Assume further that $h^{1, 1}(N)=1$ and $N$ is locally irreducible.  Then $c_1(N)>0$, namely $N$ is Fano. Similar result holds under the assumption $\Ric^+ \le 0$.
\end{corollary}
The result similar to Corollary \ref{coro:61} holds for $\Ric^+>0$ and  $\rho(N)=1$, in view of Theorem \ref{thm:3}, Proposition \ref{prop:41-ct-Ric+} and Corollary \ref{coro:43}.

\begin{corollary}\label{cor:44}
Any compact K\"ahler manifold $(N, h)$  with quasi-positive $\Ric^+$ (or quasi-positive $\Ric_k^+$) and $\rho(N)=1$, must be rationally connected.
\end{corollary}
The same holds if $\Ric^+>0$ is replaced with $\Ric^+\ge0$ and $(N^n, h)$ is locally irreducible.
For compact K\"ahler manifolds with $\Ric^+<0$, we have the result below.

\begin{proposition} Let $(N, h)$ be a compact K\"ahler manifold with $\Ric^+<0$. Then $N$ does not admits any nonzero holomorphic  vector field.
\end{proposition}
The proof is the same as that of \cite{NZ}.  A dual version of Theorem \ref{thm:CV1} is the following result.

\begin{theorem}\label{thm:CV2}
 (i) For $(N, h)$  a compact K\"ahler manifold with quasi-positive $^d$CQB,  $$H^1(N, T'N)=\{0\}.$$ In particular, $N$ is deformation rigid in the sense that it does not admit nontrivial infinitesimal holomorphic deformation.

(ii) If compact K\"ahler manifold $(N, h)$ has $^d$CQB$_2>0$, then its Ricci curvature is $2$-positive.

(iii) If  $(N, h)$ is compact with $^d$CQB$_1>0$, then $N$ is projective  and simply-connected.
\end{theorem}
\begin{proof} For (i) one  may use the conjugate operator $\#: A^{0, 1}(T'N)\to A^{1, 0} ((T'N)^*)$ which is defined for $\phi=\phi^i_{\bar{\alpha}}dz^{\bar{\alpha}}\otimes E_i$,  with $\{E_i\}$ being a unitary frame of $T'N$, as
$$
\# \phi= \overline{\phi^{i}_{\bar{\alpha}}} dz^\alpha \otimes \overline{E}_i.
$$
Since $\#(\partial \phi)=\bar{\partial}(\#(\phi))$, it implies that $\partial^*(\# (\phi))= \#(\bar{\partial}^* \phi)$. Together $\#$ induces an isomorphism between $\mathcal{H}_{\bar{\partial}}^{p, q}(N, T'N)$ and $\mathcal{H}_{\partial}^{q, p}(N, (T'N)^*)$. To prove the result, it suffices to show that any $\psi\in \mathcal{H}_{\partial}^{1, 0}(N, (T'N)^*)$, $\psi=0$. Now we apply the Kodaira-Bochner formula for $\Delta_{\partial}$ operator,  and get for $\psi=\psi^{\bar{i}}_\alpha dz^{\alpha}\otimes \overline{E}_i$
\begin{equation}\label{eq:36}
(\Delta_{\partial} \psi)^{\bar{i}}_{\gamma} =-h^{\alpha \bar{\beta}}\nabla_{\bar{\beta}}\nabla_\alpha \psi^{\bar{i}}_{\gamma}+R^{\bar{i} \,\,\,\, \sigma}  _{\,\,  \bar{j}\gamma} \psi^{\bar{j}}_\sigma +(\Ric)_\gamma^\sigma \psi^{\bar{i}}_\sigma.
\end{equation}
Taking product with $\overline{\psi}$,  as before under the unitary frame, if $\Delta_{\partial} \psi=0$ we have that
$$
0=\int_N |\nabla \psi|^2+\int_N \left[(\Ric)_{\alpha\bar{\sigma}} \psi^{\bar{i}}_\sigma \overline{\psi^{\bar{i}}_{\alpha}}+ R_{i\bar{j}\alpha\bar{\sigma}} \psi^{\bar{j}}_\sigma \overline{\psi^{\bar{i}}_{\alpha}}\right].
$$
The claimed result follows in the similar way as in the proof of Theorem \ref{thm:CV1}.

For part (ii), for any unitary frame $\{E_i\}$, let $A$ be the rank 2 skew-symmetric transformation: $A(E_1)=\overline{E}_2$, $A(E_2)=-\overline{E}_1$, and $A(E_k)=0$ for all $k\ge 3$. Then as in the CQB$>0$ case, the second part in the expression of $^d$CQB vanishes and the first part yields $\Ric(E_1, \overline{E}_1)+\Ric(E_2, \overline{E}_2)$.

Part (iii) follows from that $^d$CQB$_1>0$ is the same as $\Ric^{+}>0$ and Theorem \ref{thm:3}.
\end{proof}

By a similar argument (comparing the Einstein constant with the smallest eigenvalue of the symmetric curvature $Q$  obtained in tables of  \cite{Itoh}) as in the proof of Theorem \ref{thm:NWZ1} we also have the following corollary concerning K\"ahler C-spaces.

\begin{theorem}\label{thm:NWZ12}Let $N^n$ be a classical K\"ahler C-space with $n\geq 2$ and $b_2=1$, or a compact exceptional Hermitian symmetric space with $n\ge 2$. Then the (unique up to constant multiple) K\"ahler-Einstein metric has $^d$CQB$>0$. In particular, for a classical K\"ahler C-space $N$ with $b_2=1$, $H^q(N, T'N)=\{0\}$ with $1\le q\le n$, and $N$ is deformation rigid in the sense that it does not admit nontrivial infinitesimal holomorphic deformation.
\end{theorem}
\begin{proof} To check $^d$CQB$>0$, writing $A(E_i) =A_i^t \overline{E}_t$, we  then apply an argument similar to  the case of CQB. First  decompose $A$ into $A_1+A_2$, the symmetric and the skew symmetric parts.  As in the proof of Theorem \ref{thm:NWZ1}
$$
^dCQB(A)\ge \lambda |A_1|^2+R_{i\bar{k} s\bar{t}}\overline{(A_1)^s_i}(A_1)^t_k.
$$
Here $\lambda$ is the Einstein constant of the canonical metric.
 The problem is now reduced to check that $\lambda+\nu_1>0$ with $\nu_1$ being the smallest eigenvalue of $Q$. Recall  the Hermitian symmetric linear operator  $Q$ is defined as
 $$
 Q(X\cdot Y, \overline{Z\cdot W})=R_{X\overline{Z} Y\overline{W}}
 $$
 for $X\cdot Y=\frac{1}{2}(X\otimes Y+Y\otimes X)$ and extended linearly on the space of symmetric tensors.  This quadratic curvature  was considered previously  in \cite{CV, Itoh}. We apply their results below.
  The Hermitian symmetric case again follows from Table 2 of \cite{CV}. For the nonsymmetric classical K\"ahler C-spaces, we check the condition $\lambda+\nu_1>0$   as follows. Note that in \cite{ChauTam} and \cite{Itoh} the same normalization for the canonical metric was used.
For $(B_r, \alpha_i)_{r\ge 3, 1<i<r}$, $\lambda=2r-i$. According to Table 4 of \cite{Itoh} $\nu_1=-2(r-i)+1$ or $-2$. Since $2r\ge 2i+2$, clearly $2r-i>2$. Also $2r-i-2r+2i+1=i+1>0$. This implies  the result for both cases of $\nu_1$, namely $\nu_1=-2(r-i)+1$ and $\nu_1=-2$.

For $ (C_r, \alpha_i)_{r\ge 3, 1<i<r}$, $\lambda=2r-i+1$. According to Table 7 of \cite{Itoh}, $\nu_1=-2(r-i+1)$. Hence $\lambda+\nu_1=i-1>0$ for $i\ge 2$. This verifies the result.

For $(D_r, \alpha_i)_{r\ge 4, 1< i< r-1}$, $\lambda=2r-i-1$.  According to Table 10 of \cite{Itoh} $\nu_1=-2(r-i)+2$ or $-2$. Since $2r-i-3\ge i-1>0$ and $2r-i-1-2r+2i+2=i+1>0$, this also verifies the result.

This proved the $H^1(N, T'N)=\{0\}$.  For $q>1$, the argument of \cite{CV} implies that one only needs to check
that $\lambda+\frac{q+1}{2q}\nu_1>0$. This is a consequence of the $q=1$ case above.
\end{proof}

 For the exceptional space $(F_4, \alpha_4)$ since $\lambda=11/2$ and $\nu_1=-5$, the above result also holds. Hence it should not be surprising that the  result in the corollary  holds for the rest ($22$ of them total)  exceptional K\"ahler C-spaces. The deformation rigidity result above holds infinitesimally. It would be interesting to see for a deformation with each fiber  except the central is assumed to be biholomorphic to a fixed manifold with  $^d$CQB$>0$, whether or not  the central fiber must be  the same manifold (as the main theorem of  \cite{Siu}). If express $A=\sum_{\ell=1}^k \overline{X}_\ell\otimes \overline{Y}_\ell$, we have that  $^d$CQB$_k>0$ if and only if
 \begin{equation}
\label{eq:dcqb-k}
\sum_{\ell, j=1}^k \Ric( X_\ell, \overline{X}_j) \langle Y_\ell, \overline{Y}_j\rangle +R(X_\ell, \overline{X}_j, Y_\ell, \overline{Y}_j)>0,\quad  \forall \sum_{\ell=1}^k\overline{X}_\ell\otimes \overline{Y}_\ell \ne 0.
 \end{equation}
 Non locally Hermitian symmetric examples of compact K\"ahler manifold with $^d$CQB$<0$ have been constructed in \cite{NZ-cross}.

  Questions similar to those in (Q1) can be asked for  $^d$CQB.   We also add the following question {\it (Q2): For which $k\in \{1, \cdots, n\}$, $\Ric_k, \Ric^{\perp}_k, \Ric^{+}_k$,  CQB$_k\ge 0$, and $^d$CQB$_k\ge 0$ are preserved under the K\"ahler-Ricci flow? }

  Concerning this,  some Ricci-flow invariant cones have been constructed recently in \cite{NZ-cross}. It is also know that $\Ric^{\perp}_2$ is preserved by the K\"ahler-Ricci flow.

 \section{ Appendix-Estimates on the harmonic $(1,1)$-forms of low rank}

  Here we prove a vanishing theorem for harmonic $(1,1)$-forms of low rank related to the condition QB$_k>0$ introduced earlier. This is particularly relevant given that in \cite{NZ-cross} examples of arbitrary large $b_2$ was constructed with CQB$>0$ (in particular with $\Ric^\perp>0$).  First recall that
  $$ QB_R(A)=\sum_{\alpha, \beta=1}^nR(A(E_\alpha), \overline{A(E_\alpha)}, E_\beta, \overline{E}_\beta)-R(E_\alpha, \overline{E}_\beta, A(E_\beta), \overline{A(E_\alpha)})
 $$ vanishes for $A=\lambda \id $.  Hence when define QB$_k(A)>0$ we require the above expression positive for $A$ in  $S^2(\mathbb{C}^n)\setminus \{\lambda \id\}$,  and that $A$ has rank not greater than $k$. The space of harmonic $(1,1)$-forms  $\mathcal{H}_{\bar{\partial}}^{1,1}$ can be decomposed further. First we observe that an $(1,1)$-form $\Omega =\sqrt{-1} A_{i\bar{j}}dz^i\wedge dz^{\bar{j}}$ can be decomposed as
 $$
 \Omega =\Omega_1-\sqrt{-1}\Omega_2=\frac{\sqrt{-1}}{2} B_{i\bar{j}}dz^i\wedge dz^{\bar{j}} -\sqrt{-1} ( \frac{\sqrt{-1}}{2}C_{i\bar{j}}dz^i\wedge dz^{\bar{j}})
 $$
 with
 $$
 B_{i\bar{j}}=A_{i\bar{j}}+\overline{A_{j\bar{i}}};  \quad \quad C_{i\bar{j}}=\sqrt{-1}\left(A_{i\bar{j}}-\overline{A_{j\bar{i}}}\right).
 $$
 If $\Omega$ is harmonic, then $\partial \Omega =\bar{\partial}\Omega =0$. It can be verified that $\Omega_1$ and $\Omega_2$ are both harmonic (cf. Theorem 5.4 in Chapter 3 of \cite{MK}). This shows that $\Omega$ can be decomposed into the sum of a Hermitian symmetric one with $-\sqrt{-1}$ times another Hermitian symmetric one. Namely $\mathcal{H}_{\bar{\partial}}^{1,1}=\mathcal{H}_{\bar{\partial}, s}^{1,1}-\sqrt{-1}\mathcal{H}_{\bar{\partial}, s}^{1,1}$, where $\mathcal{H}_{\bar{\partial}, s}^{1,1}$ is the spaces of harmonic $\Omega$ with $(A_{i\bar{j}})$ being Hermitian symmetric.  Within $\mathcal{H}_{\bar{\partial}, s}^{1,1}$ we consider $\mathcal{H}_{\bar{\partial}, s}^{1,1}\setminus \{\mathbb{C}\omega\}$. To prove $b_2=1$ under the assumption $QB>0$, it suffices to show that $\mathcal{H}_{\bar{\partial}, s}^{1,1}\setminus \{\mathbb{C}\omega\}=\{0\}$. We can stratify the space into ones with rank bounded from above. Let $\mathcal{H}^{1,1}_{s, k}$ denote the subspace of $\mathcal{H}_{\bar{\partial}, s}^{1,1}$ which consists of $\Omega=\frac{\sqrt{-1}}{2}A_{i\bar{j}}dz^i\wedge dz^{\bar{j}}$ with $(A_{i\bar{j}})$ being Hermitian symmetric and  of rank no greater than $k$ everywhere on $N$.  The following result can be shown.

 \begin{theorem} Assume that $(N^n, g)$ is a compact K\"ahler manifold with quasi-positive QB$_k$ with $k<n$. Then
 $
 \mathcal{H}^{1,1}_{s, k}(N)=\{0\}.
 $
 In particular, $\Ric^\perp>0$ implies that $\mathcal{H}^{1,1}_{s, 1}(N)=\{0\}$.
 \end{theorem}
\begin{proof} Assume that $\Omega$ is a nonzero element in $\mathcal{H}^{1,1}_{s, k}(N)$.  Applying the $\Delta$ operator  to $\|\Omega\|^2$,  by Kodaira-Bochner formula  we have that
$$
\frac{1}{2}\left(\nabla_\gamma \nabla_{\bar{\gamma}}+\nabla_{\bar{\gamma}}\nabla_\gamma \right)\|\Omega\|^2(x)=\|\nabla_{\gamma} \Omega\|^2(x)+\|\nabla_{\bar{\gamma}} \Omega\|^2(x)+2QB(\Omega)(x).
$$
Integrating on $N$ we have that
$$
0=\int_N \left[\|\nabla_{\gamma} \Omega\|^2(x)+\|\nabla_{\bar{\gamma}} \Omega\|^2(x)\right]\, d\mu(x)+2\int_N QB(\Omega)(x)\, d\mu(x) >0.
$$
The last strictly inequality is due to the fact that by the unique continuation we know at a neighborhood $U$ where QB$_k>0$, and $\Omega$ is identically zero. The contradiction implies that $ \Omega \equiv 0$.
\end{proof}

For any holomorphic line bundle $L$ over $N$ with a Hermitian metric $a$, its first Chern form $c_1(L,a)=-\frac{\sqrt{-1}}{2}\partial \bar{\partial} \log a$ is a Hermitian symmetric $(1,1)$-form. If $\eta$ is the harmonic representative of $c_1(L, a)$, then  $\eta$ is Hermitian symmetric by the uniqueness of the Hodge decomposition and K\"ahler identities (cf. \cite{MK}, Chapter 3). The following is a simple observation towards possible topological meanings of  the rank of $\eta$ (the minimum $k$ such that $\eta\in \mathcal{H}^{1,1}_{\bar{\partial}, k}$, denoted as $rk(L)$).

\begin{proposition} Recall that the numerical dimension of $L$ is defined as
$$
nd(L)=\max \{ k=0, \cdots, n: c_1(L)^k\ne 0\}.
$$
Then $rk(L)\ge nd(L)$.
\end{proposition}

 The proof of the above theorem also shows that  if QB$_k\ge 0$, then any element in $\mathcal{H}^{1,1}_{s, k}(N)$ must be parallel. Thus we have the dimension estimate:
$$
\dim(\mathcal{H}^{1,1}_{s, k}(N))\le k^2.
$$
In fact the existence of a non-vanishing $(1,1)$-form of rank at most $k$ has a  strong implication due to the De Rham decomposition.

\begin{corollary} Assume that QB$_k\ge 0$ and $\mathcal{H}^{1,1}_{s, k}(N)\ne \{0\}$. Then $N$ must be locally reducible. In particular, if $N$ is locally irreducible and $\Ric^\perp\ge 0$, then $\mathcal{H}^{1,1}_{s, 1}(N)=\{0\}$.
\end{corollary}
\begin{proof} By the above, we know that the nonzero $\Omega\in \mathcal{H}^{1,1}_{s, k}(N)$ must be parallel. Its null space is invariant under the parallel transport. This provides a nontrivial parallel distribution, hence the local splitting.
\end{proof}

The product example $\mathbb{P}^2\times \mathbb{P}^2$, which satisfies $\Ric^\perp>0$ and supports non-trivial rank 2 harmonic $(1,1)$-forms, shows that the above result is sharp for $\Ric^\perp>0$. Irreducible examples of dimension greater than 4 were  constructed via the projectivized bundles in \cite{NWZ}.

\section*{Acknowledgments} {}

 The author would like to thank James McKernan for helpful discussions, L.-F. Tam and F. Zheng and for their interests. F. Zheng read the draft carefully, spotted a discrepancy  and made  helpful suggestions. The author also thank B. Wilking for  the  connection between $(2k-1)$-Ricci and $\Ric_k$.  The first version of the paper was completed during the author's visit of CUHK, Fuzhou Normal Univ., Peking Univ., SUST and Xiamen Univ., in December 2018. He thanks these institutions for the hospitality and the referee for very help comments.


\begin{thebibliography}{A}

\bibitem{AHZ} A. Alvarez, G. Heier, and F. Zheng, \textit{On projectivized vector bundle and positive holomorphic sectional curvature.} Proc. Amer. Math. Soc., \textbf{146} (2018), 2877--2882.

 \bibitem{ABCKT} J. Amor\'os, M. Burger, K. Corlette, D. Kotschick, and D. Toledo, \textit{ Fundamental groups of compact K\"ahler manifolds.} Mathematical Surveys and Monographs, \textbf{44}. American Mathematical Society, Providence, RI, 1996.

     \bibitem{BCW} R. Bamler, E. Cabezas-Rivas and B. Wilking, \textit{The Ricci flow under almost non-negative curvature conditions.} Invent. Math. \textbf{217}(2019), 95--126.

 \bibitem{Bishop}  R. Bishop and R. Crittenden, \textit{Geometry of manifolds.}  Academic Press,  1964.

\bibitem{Bott} R. Bott, \textit{ Homogeneous vector bundles.} Ann. of Math., \textbf{66} (1957), 203--248.

\bibitem{CV} E. Calabi and E.  Vesentini, \textit{ On compact, locally symmetric K\"ahler manifolds.} Ann. of Math.\textbf{71}(1960),  472--507.

\bibitem{Cam} F.  Campana, \textit{ Connexit\'e rationnelle des vari\'et\'es de Fano.} (French) [Rational connectedness of Fano manifolds] Ann. Sci. \'Ecole Norm. Sup. \textbf{25} (1992), no. 5, 539--545.

\bibitem{CDP} F.  Campana, J.-P.  Demailly and T. Peternell, \textit{ Rationally connected manifolds and semipositivity of the Ricci curvature.} Recent advances in algebraic geometry, 71--91, London Math. Soc. Lecture Note Ser., \textbf{417}, Cambridge Univ. Press, Cambridge, 2015.

\bibitem{CP} F. Campana and T.  Peternell, \textit{ Projective manifolds whose tangent bundles are numerically effective.} Math. Ann. \textbf{289}(1991), no. 1, 169--187.

 \bibitem{CT}A.  Chau and L.-F. Tam, \textit{On quadratic orthogonal bisectional curvature.}  J. Diff. Geom. \textbf{92} (2012), no. 2, 187--200.

\bibitem{ChauTam} A. Chau and L.-F.  Tam, \textit{ K\"ahler C-spaces and quadratic bisectional curvature.} J. Diff. Geom., \textbf{94} (2013), no. 3, 409--468.

\bibitem{Deb} O. Debarre, \textit{ Higher dimensional algebraic geometry.} Springer-Verlag,  2001.

\bibitem{Fede} H. Federer, \textit{ Geometric measure theory. }  Springer-Verlag, 1969.

\bibitem{FN} A. Fr\"olicher and A. Nijenhuis, \textit{ A theorem of stability of complex structures.} Proc. Nat. Acad. Sci., U.S. A., \textbf{43} (1957), 239--241.

\bibitem{Gri} P. A.  Griffiths, \textit{ Hermitian differential geometry, Chern classes, and positive vector bundles.} Global Analysis (Papers in Honor of K. Kodaira) pp. 185–-251 Univ. Tokyo Press,  1969.


\bibitem{HeierWong}
G. Heier and B. Wong, \textit{ On projective K\"ahler manifolds of partially positive curvature and rational connectedness.} ArXiv:1509.02149.

\bibitem{Hir} F.
Hirzebruch, \textit{
Topological methods in algebraic geometry.}
 Springer-Verlag,  1966.


\bibitem{Hitchin}
N. Hitchin, \textit{ On the curvature of rational surfaces.} In Differential Geometry ({\em Proc. Sympos. Pure Math., Vol XXVII, Part 2, Stanford University, Stanford, Calif., 1973}), pages 65--80. Amer. Math. Soc., 1975.

\bibitem{Itoh} M. Itoh, \textit{On curvature properties of K\"ahler C-spaces.}  J.  Math. Soc. Japan,  \textbf{30} (1978), no. 1, 39--71.

		
\bibitem{Ko} S. Kobayashi, \textit{ On compact K\"ahler manifolds with positive Ricci tensor.} Ann. of Math. \textbf{74} (1961), 570--574.


 \bibitem{Ko2} S. Kobayashi,    \textit{Differential geometry of complex vector bundles.} Princeton University Press, 1987.

\bibitem{Kod} K. Kodaira, \textit{
On K\"ahler varieties of restricted type (an intrinsic characterization of algebraic varieties).}
Ann. of Math.  \textbf{60} (1954), 28--48.


\bibitem{KS} K. Kodaira and D. C. Spencer, \textit{ On deformations of complex analytic structures. I and II.} Ann. of Math. \textbf{67}(1958), 328--460.


\bibitem{KMM} J. Koll\'ar, Y.  Miyaoka and S.  Mori, \textit{ Rational connectedness and boundedness of Fano manifolds.} J. Diff. Geom. \textbf{36} (1992), no. 3, 765--779.

\bibitem{La12} R.    Lazarsfeld, \textit{ Positivity in algebraic geometry. I\& II.}   Springer-Verlag, 2004.


\bibitem{LN} X. Li and L. Ni, \textit{ K\"ahler-Ricci shrinkers and ancient solutions with nonnegative orthogonal bisectional curvature.} Jour. Math.  Pures  Appl., \textbf{138}(2020), 28--45.

\bibitem{Mi-Pal} V. Miquel and V.  Palmer, \textit{ Mean curvature comparison for tubular hypersurfaces in K\"ahler manifolds and some applications.} Compositio Math. \textbf{86} (1993), no. 3, 317--335.



\bibitem{Mori} S.  Mori, \textit{ Projective manifolds with ample tangent bundles.}  Ann. of Math. \textbf{110} (1979), no. 3, 593--606.


 \bibitem{MK}  J.   Morrow and K.  Kodaira, \textit{ Complex manifolds.} Holt. Rinehart and Winston,  1971.

\bibitem{Nadel} A. Nadel, \textit{ The boundedness of degree of Fano varieties with Picard number one.} J. Amer. Math. Soc. \textbf{4} (1991), no. 4, 681--692.

\bibitem{Nak} S. Nakano, \textit{ On complex analytic vector bundles.} J. Math. Soc. Japan \textbf{7} (1955), 1--12.


\bibitem{N}L. Ni, \textit{ Liouville  theorems and a Schwarz Lemma for holomorphic mappings between K\"ahler manifolds.} Comm. Pure Appl. Math., to appear.



 \bibitem{Ni-ijm} L. Ni, \textit{ General Schwarz lemmata and their applications.} Internat. J. Math. \textbf{30} (2019), no. 13, 1940007, 17 pp.


\bibitem{NT2}		L. Ni and L.-F. Tam, \textit{Poincar\'e-Lelong equation via the Hodge-Laplace heat equation.}  Compositio Math. \textbf{149} (2013), 1856--1870.
		
		

\bibitem{NWZ} L. Ni, Q. Wang and F. Zheng, \textit{Manifolds with positive orthogonal Ricci curvature.} ArXiv preprint:1806.10233.



\bibitem{NZ} L. Ni and F. Zheng, \textit{ Comparison and vanishing  theorems for K\"ahler manifolds.}  Calc. Var. Partial Differential Equations, {\bf 57} (2018), no. 6, Art. 151, 31 pp.


\bibitem{Ni-Zheng2} L. Ni and F. Zheng, \textit{ Positivity and Kodaira embedding theorem.} ArXiv preprint:1804.09696.

\bibitem{NZ-con}     L. Ni and F. Zheng, \textit{    On orthogonal Ricci curvature.} Contemporary Math., Advances in Complex Geometry, \textbf{735} (2019), 203--215.

\bibitem{NZ-cross}     L. Ni and F. Zheng, \textit{ K\"ahler manifolds and cross quadratic bisectional curvature.} ArXiv preprint:1903.02701.


\bibitem{Po} A.V.    Pogorelov,\textit{ The Minkowski multidimensional problem.} Halsted Press [John Wiley \& Sons],  1978.


\bibitem{Siu} Y.-T. Siu, \textit{ Nondeformability of the complex projective spaces.} J. Reine Angew. Math. \textbf{399}(1989), 208--219.

 \bibitem{SY} Y.-T. Siu and S.-T. Yau,
 \textit{ Compact K\"ahler manifolds of positive bisectional curvature.} Invent. Math. \textbf{59} (1980), no. 2, 189--204.


\bibitem{Tian}  G.    Tian, \textit{Canonical metrics in K\"ahler geometry.}  Birkh\"auser Verlag, 2000.



 \bibitem{Tsu}  Y.   Tsukamoto,\textit{ On K\"ahlerian manifolds with positive holomorphic sectional curvature. } Proc. Japan Acad. \textbf{33} (1957), 333--335.
	

\bibitem{Whit} H.
Whitney, \textit{ Geometric integration theory.} Princeton University Press,  1957.

\bibitem{Wu} H. Wu, \textit{ Manifolds of partially positive curvature.} Indiana Univ. Math. J. \textbf{36} (1987), no. 3, 525--548.

 \bibitem{WYZ} D. Wu, S.T. Yau and  F. Zheng, \textit{ A degenerate Monge Amp\`ere equation and the boundary classes of K\"ahler cones.} Math. Res. Lett. \textbf{16}  (2009), no.2, 365--374.

\bibitem{YZ} B. Yang and F. Zheng, \textit{ Hirzebruch manifolds and positive sectional curvature.} Ann. Inst. Fourier (Grenoble), \textbf{69} (2019), no. 6, 2589--2634.
		


\bibitem{XYang}
X. Yang, \textit{ RC-positivity, rational connectedness, and Yau's conjecture.} Cambridge J. Math. \textbf{6}(2018), 183--212.


\bibitem{Yang18}
X. Yang, \textit{ RC-positive metrics on rationally connected manifolds.}   ArXiv preprint:1807.03510.


\bibitem{Yau} S.-T. Yau, \textit{ On Ricci curvature of a compact K\"ahler manifold and the complex Monge-Amp\`ere equation, I,} Comm. Pure and Appl. Math. \textbf{31} (1978), 339--411.

\bibitem{Zheng} F. Zheng, \textit{
Complex differential geometry.} AMS/IP Studies in Advanced Mathematics,  2000.

\end{thebibliography}
\end{document}